\documentclass{amsart}

\usepackage{mathtools}
\usepackage{amsmath}
\usepackage{amsthm}
\usepackage{amssymb}
\usepackage{enumitem}
\usepackage{booktabs}
\usepackage{iftex}
\ifptex
  
\fi
\usepackage{tikz}
\usetikzlibrary{arrows.meta,decorations.pathreplacing}
\ifptex
  \usepackage[dvipdfmx,hidelinks]{hyperref}
\else
  \usepackage[hidelinks]{hyperref}
\fi
\usepackage{cleveref}

\numberwithin{equation}{section}

\newcommand{\R}{\mathbb{R}}
\newcommand{\B}{\mathbb{B}}
\newcommand{\BetaBall}{\mathrm P\mathbb B^n_{\beta_n}}
\newcommand{\Pyr}{\mathcal{P}}
\newcommand{\Ple}{\mathcal{P}_{\le 1}}
\newcommand{\PGauss}{\mathcal{P}_{\Gamma^\infty}}
\newcommand{\dH}{d_{\mathbb H}}
\newcommand{\dHaus}[1]{(#1)_H}
\newcommand{\dconc}{d_{\mathrm{conc}}}

\theoremstyle{plain}
\newtheorem{thm}{Theorem}[section]
\newtheorem{prop}[thm]{Proposition}
\newtheorem{lem}[thm]{Lemma}
\newtheorem{cor}[thm]{Corollary}
\crefname{thm}{Theorem}{Theorems}
\crefname{prop}{Proposition}{Propositions}
\crefname{lem}{Lemma}{Lemmas}
\crefname{cor}{Corollary}{Corollaries}

\theoremstyle{definition}
\newtheorem{defn}[thm]{Definition}
\crefname{defn}{Definition}{Definitions}

\theoremstyle{remark}
\newtheorem{remark}[thm]{Remark}
\crefname{remark}{Remark}{Remarks}

\crefname{section}{Section}{Sections}
\crefname{figure}{Figure}{Figures}
\crefname{table}{Table}{Tables}

\setlist[enumerate]{leftmargin=2em,itemsep=0.25em,topsep=0.25em}
\setlist[itemize]{leftmargin=2em,itemsep=0.25em,topsep=0.25em}
\title[Poincar\'e Beta Balls]{Poincar\'e Beta Balls: Radial Laws, Hyperbolic Shell Metric Transforms, and Critical-Limit Regimes}
\author{Shigeaki Yokota}
\thanks{Mathematical Institute, Graduate School of Science, Tohoku University, Sendai, Japan\\Corresponding author: \texttt{shigeaki.yokota.t4@dc.tohoku.ac.jp}}
\date{}

\begin{document}

\begin{abstract}
On a high-dimensional Poincar\'e ball, a Euclidean beta-type radial measure concentrates near a sphere, but hyperbolic distance amplifies the surviving radial spread, so the usual shell reduction loses part of the limiting metric. In each regime of the balance between this radial width and amplified angular separation, we determine the weak pyramid limit of the spaces rescaled at the order at which the transition between concentration and dissipation occurs. The four critical-limit regimes are the shifted-exponential railway pyramid, the diameter-at-most-one regime, the nonlinear Gaussian-transform regime, and the Gaussian regime. Each finite-branch railway mm-space generating the first pyramid has branches from a common center, a shifted exponential distribution along them, and paths between different branches through the center. We also give a sharp criterion for convergence to the diameter-at-most-one pyramid. The resulting classification identifies the precise critical scales and clarifies when angular and radial fluctuations contribute simultaneously to the limiting geometry.
\end{abstract}

\maketitle

\section{Introduction}
\label{sec:introduction-draft}

The unit spheres $S^n$, equipped with their geodesic distances and normalized Riemannian volume measures, provide the classical model of concentration of measure: as $n\to\infty$, every $1$-Lipschitz function becomes nearly constant on almost all of $S^n$, uniformly over the choice of the function. Gromov developed the geometry of mm-spaces to treat this phenomenon. The relevant objects are complete separable metric spaces equipped with full-support Borel probability measures, called mm-spaces. We write $X$ for one such object and $\mathcal X$ for the set of mm-isomorphism classes. Gromov introduced the observable distance \cite[pp.~199--200]{gromov2007met}, whose convergence is concentration of mm-spaces, and the box distance \cite[pp.~116--117]{gromov2007met}, which gives a strictly finer topology on mm-spaces. The modern formulation used here is described in \cite[Introduction, Proposition~5.5, Sections~4.6 and~5.5]{shioya2016mmg}. We denote these distances by $\dconc$ and $\Box$, respectively. Gromov also introduced the pyramid construction and terminology as a natural compactification, including the pyramid associated with each mm-space \cite[pp.~212--216]{gromov2007met}. The modern formulation used here is described in \cite[Chapter~6]{shioya2016mmg} and includes box closedness, a condition added later to obtain the Hausdorff property \cite[Definition~2.21 and the paragraph following it]{ozawa2015limit}. An mm-space $X$ \emph{dominates} $Y$ if there is a measure-preserving $1$-Lipschitz map from $X$ to $Y$. The pyramid associated with $X$ is
\[\Pyr(X)\coloneqq\{Y\in\mathcal X\mid X\text{ dominates }Y\}.\]
A pyramid is a nonempty $\Box$-closed subfamily of $\mathcal X$ that is downward closed under domination and in which any two members have a common dominating member. A sequence of pyramids converges \emph{weakly} when it converges as closed subsets of $(\mathcal X,\Box)$ in the sequential Painlev\'e--Kuratowski sense, also called weak Hausdorff convergence in mm-space theory. We write $\Pi$ for the set of all pyramids and $d_\Pi$ for the standard metric on $\Pi$ that metrizes this weak convergence. Then $(\Pi,d_\Pi)$ is compact \cite[Definition~4.5 and Theorem~4.6]{shioya2022sugaku}.

Spheres are the classical concrete examples of concentration. Let $\gamma^k$ be the standard Gaussian probability measure on $\R^k$ and $\Gamma^k\coloneqq(\R^k,|\cdot|,\gamma^k)$. The \emph{Gaussian pyramid} is
\[\PGauss\coloneqq\overline{\bigcup_{k\ge1}\Pyr(\Gamma^k)}^{\,\Box}.\]
As $n\to\infty$, the pyramids associated with the $n$-dimensional spheres $S^n(\sqrt n)$ of radius $\sqrt n$, equipped with the restricted Euclidean distance, converge weakly to $\PGauss$ \cite[Theorem~7.40]{shioya2016mmg}. This result extends to high-dimensional solid ellipsoids: suitably rescaled ellipsoids, equipped with normalized volume measures, converge weakly to an infinite-dimensional Gaussian space, and the conditions under which this convergence is in the box topology are characterized \cite[Theorems~1.1 and~1.2]{kazukawa-shioya2020ellipsoids}.

A further example is supplied by constant-curvature cones. For $\kappa\in\R$, put $I_\kappa=[0,\pi/\sqrt\kappa]$ if $\kappa>0$ and $I_\kappa=[0,+\infty)$ if $\kappa\le0$. The $\kappa$-cone over a metric space $X$ is the quotient of $I_\kappa\times X$ obtained by collapsing the endpoints of $I_\kappa$. Writing $rx$ for the class of $(r,x)$, its radial distance is determined by
\[c_\kappa(d(rx,sx'))=c_\kappa(r)c_\kappa(s)+\kappa s_\kappa(r)s_\kappa(s)\cos(\min\{d_X(x,x'),\pi\}),\]
where $s_\kappa$ solves $s_\kappa''+\kappa s_\kappa=0$, $s_\kappa(0)=0$, $s_\kappa'(0)=1$, and $c_\kappa=s_\kappa'$. For $\kappa=0$, the limiting formula is
\[d(rx,sx')^2=r^2+s^2-2rs\cos(\min\{d_X(x,x'),\pi\}).\]
Thus positive, zero, and negative $\kappa$ give the spherical, Euclidean, and hyperbolic cosine laws, respectively. A radial probability measure on $I_\kappa$ and the measure of $X$ induce the cone measure by pushforward under the quotient map. For fixed $\kappa$, weak convergence of the underlying pyramids and radial measures implies weak convergence of these $\kappa$-cones. Applied to the $n$-dimensional Cauchy space with exponent $\beta>0$ and the generalized Cauchy distribution, this gives concentration to a half-line after rescaling by $n^{-1/2}$ as $n\to\infty$ \cite[Definition~1.1 and Theorems~1.2 and~1.4]{esaki-kazukawa-mitsuishi2024cones}. These examples all use Euclidean space or a subset of it with the restricted Euclidean distance. In contrast, the example in this paper arises from the interaction between the hyperbolic distance on the Poincar\'e ball and a Euclidean radial density. Its pyramid limit is the shifted-exponential railway pyramid.

Let
\[\B^n\coloneqq\{x\in\R^n\mid |x|<1\}\]
be the $n$-dimensional Poincar\'e ball with hyperbolic distance $\dH$. For $n\ge1$ and $\beta>0$, define
\[d\mu_{n,\beta}(x)\coloneqq C_{n,\beta}(1-|x|^2)^{\beta-1}\,dx,\qquad C_{n,\beta}\coloneqq\frac{\Gamma(n/2+\beta)}{\pi^{n/2}\Gamma(\beta)},\]
where $dx$ denotes the Lebesgue measure on $\R^n$. Given a positive sequence $(\beta_n)$, set
\[\BetaBall\coloneqq(\B^n,\dH,\mu_{n,\beta_n}).\]
The upright tag $\mathrm P$ marks the Poincar\'e ball, whose distance is hyperbolic.

The standard shell reduction for a high-dimensional radial measure replaces the radial coordinate by a representative radius after the measure concentrates near a Euclidean sphere. For $\BetaBall$, however, the hyperbolic distance amplifies deviations from that sphere. If the radial width survives at the critical scale, freezing the distance at the representative sphere discards part of the limiting metric.

Let $\Ple$ be the pyramid of all mm-spaces of diameter at most $1$. Even the simplest choice $\beta_n\equiv1$ has the nontrivial weak limit
\[\Pyr\bigl((\log n)^{-1}\mathrm P\mathbb B^n_1\bigr)\longrightarrow\Ple.\]
Intuitively, $\Ple$ represents an infinite discrete space of diameter $1$.

Retaining both the angular amplification and the radial width requires the representative-shell amplification parameter and its logarithm, defined for $n\ge1$ and $\beta>0$ by
\[\mathcal A_{n,\beta}\coloneqq\frac{\sqrt{n+2\beta}}{\beta},\qquad \mathcal L_{n,\beta}\coloneqq2\log\mathcal A_{n,\beta}.\]

When both effects survive, the shifted-exponential railway pyramid is generated by finite-branch railway mm-spaces with a parameter $\lambda\in[0,+\infty)$. Each such space has several half-line branches issuing from a common center, with the coordinate $s$ along a branch distributed as $\lambda+\operatorname{Exp}(1)$ (the exponential distribution of rate one shifted by $\lambda$), which we denote by $\nu_\lambda$. A point with coordinate $s$ lies at distance $s-\lambda/2$ from the center, so two points with coordinates $s$ and $t$ on different branches are joined through the center at distance $s+t-\lambda$.

More precisely, fix $m\ge1$ and positive weights $p_1,\ldots,p_m$ such that $\sum_{i=1}^m p_i=1$. Equip $[\lambda,+\infty)\times\{1,\ldots,m\}$ with the measure
\[\nu_\lambda\otimes\sum_{i=1}^m p_i\delta_i\]
and the pseudometric
\[
  d((s,i),(t,j)) \coloneqq \begin{cases}
    |s-t|, & i=j,\\
    s+t-\lambda, & i\ne j
  \end{cases}.
\]
For $\lambda>0$, $d$ is a genuine metric, while for $\lambda=0$ and $m\ge2$ the points $(0,i)$ and $(0,j)$ on distinct branches have distance zero and become their common center in the metric quotient, which is why we call $d$ a pseudometric.
Its metric quotient is denoted by $T_{m,\lambda}(p_1,\ldots,p_m)$. Allowing every finite number of branches and all positive branch weights gives the pyramid
\[
  \mathcal P_\lambda^\star
  \coloneqq \overline{
    \bigcup_{m\ge1,\ p_i>0,\ \sum_i p_i=1}
    \Pyr(T_{m,\lambda}(p_1,\ldots,p_m))
  }^{\,\Box}.
\]
Here the bar denotes closure with respect to the box distance $\Box$, recalled in \Cref{def:common-box-concentration}.

\begin{thm}[Railway-pyramid regime]
\label{thm:phase-diagram-draft}
If $\beta_n\to0$ and $\beta_n\log n\to\lambda\in[0,+\infty)$ as $n\to\infty$, then
\[\Pyr(\beta_n\BetaBall)\longrightarrow\mathcal P_\lambda^\star\]
weakly.
\end{thm}

In \Cref{fig:intro-star-tree}, the distance from the center $O$ to the first point in the support on each branch is $\lambda/2$, and the support continues from that point along a half-line. When $\lambda=0$, this initial gap disappears and the supported parts of the branches begin at the center.

\begin{figure}[htbp]
  \centering
  \begin{tikzpicture}[
      x=1.25cm,
      y=1.25cm,
      >=Latex,
      base segment/.style={draw=black!65, densely dashed, line width=0.9pt},
      usable ray/.style={draw=black, line width=1.05pt, -{Latex[length=2.2mm]}},
      marked point/.style={circle, fill=black, inner sep=1.7pt},
      cross route/.style={draw=black!30, line width=5pt,
        line cap=round, line join=round}
    ]
    \coordinate (O) at (0,0);
    \coordinate (P1) at (1.039,0.600);
    \coordinate (P2) at (1.200,0);
    \coordinate (P3) at (1.039,-0.600);
    \coordinate (S1) at (2.598,1.500);
    \coordinate (S2) at (3.000,0);
    \coordinate (T2) at (4.000,0);
    \coordinate (T3) at (3.464,-2.000);
    \coordinate (E1) at (4.503,2.600);
    \coordinate (E2) at (5.200,0);
    \coordinate (E3) at (4.503,-2.600);

    % Highlight the cross-branch route before redrawing the branch styles on top.
    \draw[cross route] (S1) -- (O) -- (T3);
    \draw[base segment] (O) -- (P1);
    \draw[base segment] (O) -- (P2);
    \draw[base segment] (O) -- (P3);
    \draw[usable ray] (P1) -- (E1);
    \draw[usable ray] (P2) -- (E2);
    \draw[usable ray] (P3) -- (E3);

    \node[marked point, label={[xshift=-2pt,yshift=4pt]above left:$O$}] at (O) {};
    \node[marked point, label={[xshift=-2pt,yshift=1pt]above:$P_1$}] at (P1) {};
    \node[marked point, label=below:$P_2$] at (P2) {};
    \node[marked point, label={[xshift=-2pt,yshift=-1pt]below:$P_3$}] at (P3) {};

    \path (O) -- node[pos=0.52, above, sloped, fill=white, inner sep=1pt]
      {$\lambda/2$} (P1);
    \path (O) -- node[pos=0.52, above, fill=white, inner sep=1pt]
      {$\lambda/2$} (P2);
    \path (O) -- node[pos=0.52, below, sloped, fill=white, inner sep=1pt]
      {$\lambda/2$} (P3);

    \node[marked point, label=above left:{$(s,1)$}] at (S1) {};
    \node[marked point, label=below:{$(s,2)$}] at (S2) {};
    \node[marked point, label=below:{$(t,2)$}] at (T2) {};
    \node[marked point, label=below left:{$(t,3)$}] at (T3) {};

    \draw[decorate, decoration={brace, amplitude=5pt, raise=5pt}]
      (S2) -- node[midway, yshift=12pt, fill=white, inner sep=1pt]
      {$|s-t|$} (T2);

    \path (O) -- node[pos=0.73, above, sloped,
      fill=white, inner sep=1.5pt] {$s-\lambda/2$} (S1);
    \path (O) -- node[pos=0.64, below, sloped,
      fill=white, inner sep=1.5pt] {$t-\lambda/2$} (T3);
    \node[draw=black, rounded corners=2pt, fill=white,
      align=center, anchor=east, inner sep=3pt] (total)
      at (-0.45,0) {through-center distance\\$s+t-\lambda$};
    \draw[black, -{Latex[length=2mm]}] (total.east) -- (-0.12,0);

    \node[anchor=west] at (E1) {$i=1$};
    \node[anchor=west] at (E2) {$i=2$};
    \node[anchor=west] at (E3) {$i=3$};
  \end{tikzpicture}
  \caption{The finite-branch railway mm-space $T_{3,\lambda}$ as three half-line branches arranged around the common center $O$. The dashed initial segments connect $O$ to the supported part of each branch, and the highlighted path between two branches passes through $O$.}
  \label{fig:intro-star-tree}
\end{figure}
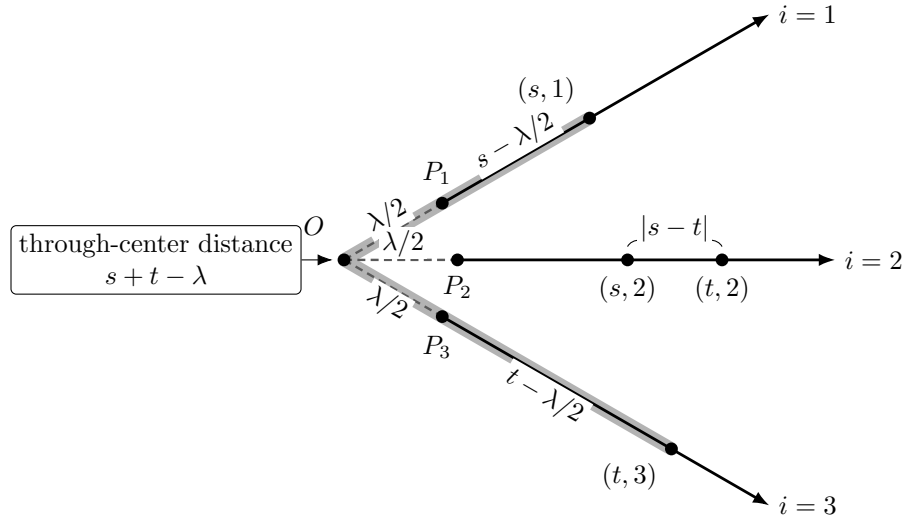

\Cref{tab:representative-regimes-draft} places the railway-pyramid theorem among the four critical-limit regimes, using $\mathcal A$ and $\mathcal L$ as abbreviations for $\mathcal A_{n,\beta_n}$ and $\mathcal L_{n,\beta_n}$. Its critical pyramids include the Gaussian pyramid $\PGauss$, generated by the finite-dimensional standard Gaussian spaces, and its hyperbolic shell metric transforms $F_a(\PGauss)$. For $a>0$, each distance $t$ in each space of $\PGauss$ is replaced by $F_a(t)\coloneqq2\operatorname{arsinh}(at/2)$, and $F_a(\PGauss)$ is the box-closed union of the pyramids associated with the resulting spaces.

\begin{table}[t]
  \centering
  \small
  \renewcommand{\arraystretch}{1.35}
  \begin{tabular}{@{}ccccc@{}}
    \toprule
    \shortstack{critical-limit\\regime}
      & parameter condition
      & scale
      & pyramid
      & theorem \\
    \midrule
    \shortstack{railway-pyramid\\regime}
      & $\beta_n\to0,\ \beta_n\log n\to\lambda\in[0,+\infty)$
      & $\beta_n$
      & $\mathcal P_\lambda^\star$
      & \Cref{thm:phase-diagram-draft} \\
    \midrule
    \shortstack{diameter-at-most-\\one regime}
      & $\beta_n/\sqrt n\to0,\ \beta_n\log n\to+\infty$
      & $\mathcal L^{-1}$
      & $\Ple$
      & \Cref{thm:leb-full} \\
    \shortstack{nonlinear\\Gaussian-transform\\regime}
      & $\sqrt n/\beta_n\to a\in(0,+\infty)$
      & $1$
      & $F_a(\PGauss)$
      & \Cref{thm:cross-full} \\
    \shortstack{Gaussian\\regime}
      & $\sqrt n/\beta_n\to0$
      & $\mathcal A^{-1}$
      & $\PGauss$
      & \Cref{thm:gauss-full} \\
    \bottomrule
  \end{tabular}
  \caption{The four critical-limit regimes, their parameter conditions, critical scales and pyramids, and the corresponding limit theorems. For a positive constant $c$, representative sequences $\beta_n$ for the four rows are $c/\log n$, $1$, $c\sqrt n$, and $n$, respectively. In the railway-pyramid regime, $c/\log n$ corresponds to $\lambda=c>0$, while $(\log n)^{-2}$ represents the case $\lambda=0$.}
  \label{tab:representative-regimes-draft}
\end{table}

Two mechanisms organize the four rows. In the three shell regimes, the beta--Gamma representation isolates a representative radius and a radial window. Once that window vanishes at the critical scale, a $\Box$-comparison transfers the problem to the representative shell, where spherical Gaussian-pyramid convergence and continuity under metric transforms give the limiting pyramid. The railway-pyramid regime retains both the radial width and angular separation. Its lower inclusion maps well-separated spherical blocks to branches, whereas its upper inclusion quantizes the angular profile of an arbitrary measurement into finitely many branches.

The angular limits in the nonlinear Gaussian-transform and Gaussian regimes combine the spherical Gaussian-pyramid convergence due to Shioya with the continuity of metric-transformed pyramids \cite[Corollary~1.4(A) and Proposition~4.1]{kazukawa2022metric} and the radial comparisons proved here. The theorem-level additions that complete the classification into four critical-limit regimes are the upper inclusion in \Cref{thm:phase-diagram-draft}, namely that every weak subsequential limit in the railway-pyramid regime is contained in $\mathcal P_\lambda^\star$, and the sharp boundary \Cref{thm:sharp-width}, which proves that convergence to $\Ple$ is equivalent to $\beta_n\log n\to+\infty$ under $\beta_n/\sqrt n\to0$.

\section{Preliminaries}
\label{sec:common-setup}

\subsection{mm-Spaces}

% v2: sec:preliminaries
\begin{defn}[mm-Space {\cite[Definition~2.8]{shioya2016mmg}}]
\label{def:common-mm-space}
A triple $(X,d_X,\mu_X)$, or simply $X$, is called an \emph{mm-space} if $(X,d_X)$ is a complete separable metric space and $\mu_X$ is a Borel probability measure with full support.
\end{defn}

\begin{defn}[mm-Isomorphism {\cite[Definition~2.9]{shioya2016mmg}}]
Two mm-spaces are called \emph{mm-isomorphic} if there is a measure-preserving bijective isometry between them. We identify mm-isomorphic spaces and write $\mathcal X$ for the set of mm-isomorphism classes.
\end{defn}

\begin{defn}[Scaling]
For $c>0$, put
\[cX\coloneqq (X,cd_X,\mu_X)\]
and call $cX$ the mm-space $X$ \emph{scaled by} $c$.
\end{defn}

\begin{defn}[Lipschitz order {\cite[Definition~2.10]{shioya2016mmg}}]
If there is a measure-preserving $1$-Lipschitz map from $X$ to $Y$, then we write $Y\prec X$ and say that $X$ \emph{dominates} $Y$. The relation $\prec$ is called the \emph{Lipschitz order}.
\end{defn}

For probabilistic notation, $X\sim\mu$ means that the random element $X$ has law $\mu$. We write $X_n\overset{\mathrm d}{\longrightarrow}X$ and $X_n\overset{\mathrm P}{\longrightarrow}X$ for convergence in distribution and convergence in probability, respectively. For Borel probability measures, $\mu_n\Rightarrow\mu$ denotes weak convergence. Throughout this paper, $\Rightarrow$ never denotes logical implication. The symbols $\mathbb E$, $\operatorname{Var}$, and $\mathbb P$ denote expectation, variance, and probability, respectively. The symbol $\operatorname{Exp}(1)$ denotes the exponential distribution of rate one, with density $e^{-t}$ on $[0,+\infty)$.

\subsection{Box and observable distances}

Closure with respect to the box distance is one of the defining properties of a pyramid. The box and observable distances are used here in the coupling formulations due to Nakajima \cite{nakajima2022coupling}. Let $(Z,d)$ be a metric space. For $z\in Z$ and $A\subset Z$, put
\[d(z,A)\coloneqq\inf_{a\in A}d(z,a).\]
For nonempty subsets $A,B\subset Z$, put
\[d(A,B)\coloneqq\inf_{a\in A,\,b\in B}d(a,b).\]
When the ambient metric is $d_X$, we write the latter distance as $d_X(A,B)$. For $A\subset Z$ and $\varepsilon\ge0$, write its closed $\varepsilon$-neighborhood as
\[\mathrm B(A,\varepsilon;d)\coloneqq\{z\in Z\mid d(z,A)\leq\varepsilon\}.\]
For subsets $A,B\subset Z$, write
\[\dHaus{d}(A,B)\coloneqq\inf\{\varepsilon\ge0\mid A\subset\mathrm B(B,\varepsilon;d),\ B\subset\mathrm B(A,\varepsilon;d)\}\]
for the Hausdorff pseudometric induced by $d$. This is a pseudometric on the collection of subsets of $Z$.

% v2: sec:preliminaries（Domination, additive-error domination, and box distance）
Let $X$ and $Y$ be mm-spaces.

\begin{defn}[Couplings]
Let $\mathcal T(\mu_X,\mu_Y)$ be the set of Borel probability measures on $X\times Y$ whose marginals are $\mu_X$ and $\mu_Y$. An element $\pi\in\mathcal T(\mu_X,\mu_Y)$ is called a \emph{coupling}, or a \emph{transport plan}.
\end{defn}

\begin{defn}[Box distance {\cite[Proposition~4.4]{nakajima2022coupling}}]
\label{def:common-box-concentration}
For a closed set $S\subset X\times Y$, put
\[
\begin{aligned}
\operatorname{dis}_{\succ}S&\coloneqq \sup_{(x,y),(x',y')\in S}\bigl(d_Y(y,y')-d_X(x,x')\bigr),\\
\operatorname{dis}S&\coloneqq \sup_{(x,y),(x',y')\in S}\left|d_X(x,x')-d_Y(y,y')\right|
\end{aligned}
\]
and define the \emph{box distance} by
\[\Box(X,Y)\coloneqq \inf_{\pi,S}\max\{1-\pi(S),\operatorname{dis}S\}.\]
Here the infimum is taken over all $\pi\in\mathcal T(\mu_X,\mu_Y)$ and all closed sets $S\subset X\times Y$. For $\varepsilon\ge0$, write $X\succ_\varepsilon Y$ if there exist $\pi\in\mathcal T(\mu_X,\mu_Y)$ and a closed set $S\subset X\times Y$ such that $\operatorname{dis}_{\succ}S\le\varepsilon$ and $\pi(S)\ge1-\varepsilon$. In particular, $X\succ_0Y$ is equivalent to $Y\prec X$ \cite[Definition~1.2 and Theorem~1.3(1)]{nakajima2024lipschitz}.
\end{defn}

\begin{defn}[Observable distance {\cite[Theorem~5.2]{nakajima2022coupling}}]
Let $\operatorname{Lip}_1(X)$ be the set of real-valued $1$-Lipschitz functions on $X$. For $\pi\in\mathcal T(\mu_X,\mu_Y)$, pull the observable families back to $X\times Y$ by the coordinate projections $\operatorname{pr}_X$ and $\operatorname{pr}_Y$:
\[
\begin{aligned}
\operatorname{pr}_X^*\operatorname{Lip}_1(X)&\coloneqq\{f\circ\operatorname{pr}_X\mid f\in\operatorname{Lip}_1(X)\},\\
\operatorname{pr}_Y^*\operatorname{Lip}_1(Y)&\coloneqq\{g\circ\operatorname{pr}_Y\mid g\in\operatorname{Lip}_1(Y)\}.
\end{aligned}
\]
For measurable functions $h,k$ on the probability space $(X\times Y,\pi)$, their \emph{Ky Fan metric} is
\[d_{\mathrm{KF}}^\pi(h,k)\coloneqq \inf\{\varepsilon\ge0\mid\pi(\{|h-k|>\varepsilon\})\le\varepsilon\}.\]
The \emph{observable distance} is
\[\dconc(X,Y)\coloneqq \min_{\pi\in\mathcal T(\mu_X,\mu_Y)}\dHaus{d_{\mathrm{KF}}^\pi}\bigl(\operatorname{pr}_X^*\operatorname{Lip}_1(X),\operatorname{pr}_Y^*\operatorname{Lip}_1(Y)\bigr),\]
and this minimum is attained \cite[Theorem~5.2]{nakajima2022coupling}.
\end{defn}

\subsection{Pyramids}

In the compactification used here, an mm-space is represented by the family of all mm-spaces it dominates.

\begin{defn}[Pyramids {\cite[Definition~2.21]{ozawa2015limit}}]
\label{def:common-pyramid}
A \emph{pyramid} is a nonempty set $\mathcal P$ of mm-isomorphism classes satisfying the following conditions.
\begin{enumerate}[label=\textup{(\roman*)}]
  \item If $X\in\mathcal P$ and $Y\prec X$, then $Y\in\mathcal P$.
  \item For any $X,Y\in\mathcal P$, there is $Z\in\mathcal P$ such that $X\prec Z$ and $Y\prec Z$.
  \item The set $\mathcal P$ is closed with respect to the box distance.
\end{enumerate}
The pyramid associated with $X$ is the family $\Pyr(X)$ introduced in the Introduction.
\end{defn}

\begin{defn}[Generated pyramids {\cite[Proposition~2.25]{esaki-kazukawa-mitsuishi2024cones}}]
A family $\mathcal F$ of mm-spaces is \emph{upward-directed} if, for any $X,Y\in\mathcal F$, there is $Z\in\mathcal F$ satisfying $X\prec Z$ and $Y\prec Z$. The \emph{pyramid generated by} $\mathcal F$ is
\[\overline{\bigcup_{X\in\mathcal F}\Pyr(X)}^{\,\Box}.\]
The coordinate projections show that the finite-dimensional standard Gaussian spaces $\{\Gamma^k\}_{k\ge1}$ introduced above form an upward-directed family. Following \cite[Definitions~7.38 and~7.39]{shioya2016mmg}, we call their generated pyramid $\PGauss$ the virtual infinite-dimensional Gaussian space.
\end{defn}

\begin{defn}[Weak convergence {\cite[Definition~2.23]{ozawa2015limit}}]
\label{def:common-pyramid-weak-convergence}
A sequence of pyramids $\mathcal P_n$ converges \emph{weakly} to $\mathcal P$ in the sense introduced above if every $X\in\mathcal P$ can be approximated by a sequence $X_n\in\mathcal P_n$ satisfying $\Box(X_n,X)\to0$. Conversely, if $n_j\to+\infty$, $X_j\in\mathcal P_{n_j}$, and $\Box(X_j,X)\to0$, then $X\in\mathcal P$. This weak topology is compact \cite[Theorem~4.6]{shioya2022sugaku}.
\end{defn}

\subsection{Phase transition property}

Observable diameter formulates the concentration side of the phase transition property, whereas separation formulates the dissipation side.

% v2: def:obsdiam-separation
\begin{defn}[Observable diameter {\cite[Definition~2.4]{ozawa2015limit}}]
Let $\nu$ be a Borel probability measure on a metric space $(Z,d)$, and let $0<\kappa<1$. Its \emph{partial diameter} at mass $1-\kappa$ is
\[\operatorname{PartDiam}(\nu;1-\kappa)\coloneqq \inf\{\operatorname{diam}A\mid A\subset Z\text{ is Borel},\ \nu(A)\ge1-\kappa\}.\]
This partial diameter is written $\operatorname{diam}(\nu;1-\kappa)$ in \cite{shioya2016mmg}.
For an mm-space $X$, define its \emph{observable diameter} by
\[\operatorname{ObsDiam}(X;-\kappa)\coloneqq \sup_{f\in\operatorname{Lip}_1(X)}\operatorname{PartDiam}(f_*\mu_X;1-\kappa).\]
For a pyramid $\mathcal P$, put
\[\operatorname{ObsDiam}(\mathcal P;-\kappa)\coloneqq \sup_{X\in\mathcal P}\operatorname{ObsDiam}(X;-\kappa).\]
For $\kappa\ge1$, define the partial diameter and observable diameter to be zero.
\end{defn}

\begin{defn}[Separation {\cite[Definition~2.8]{ozawa2015limit}}]
\label{def:common-observable-separation}
For an mm-space $X$, define its \emph{separation distance} by
\[\operatorname{Sep}(X;\kappa_0,\ldots,\kappa_N)\coloneqq \sup_{A_0,\ldots,A_N}\min_{i\ne j}d_X(A_i,A_j),\]
where the supremum is taken over Borel sets satisfying $\mu_X(A_i)\ge\kappa_i$. For a pyramid $\mathcal P$, put
\[\operatorname{Sep}(\mathcal P;\kappa_0,\ldots,\kappa_N)\coloneqq \lim_{\delta\downarrow0}\sup_{X\in\mathcal P}\operatorname{Sep}(X;\kappa_0-\delta,\ldots,\kappa_N-\delta)\]
\cite[Definition~4.3]{ozawa2015limit}. Here the limit is taken over $0<\delta<\min_i\kappa_i$. We use the displayed mass-decrement formula for pyramid separation below.
\end{defn}

% v2: sec:preliminaries（Phase transition property）; cor:critical-pyramid-ptp
\begin{defn}[L\'evy family {\cite[Definition~2.14]{shioya2016mmg}}]
\label{def:common-levy-dissipation}
A sequence of mm-spaces $(X_n)$ is a \emph{L\'evy family} if $\operatorname{ObsDiam}(X_n;-\kappa)\to0$ for every $\kappa>0$.
\end{defn}

\begin{defn}[Dissipation {\cite[Definition~8.1]{shioya2016mmg}}]
For $\delta>0$, we say that a sequence of mm-spaces $(X_n)$ $\delta$-\emph{dissipates} if, for every $N\ge1$ and every choice of positive numbers $\kappa_0,\ldots,\kappa_N$ with $\sum_i\kappa_i<1$,
\[\liminf_{n\to\infty}\operatorname{Sep}(X_n;\kappa_0,\ldots,\kappa_N)\ge\delta.\]
It \emph{infinitely dissipates} if it $\delta$-dissipates for every $\delta>0$.
\end{defn}

\begin{defn}[Phase transition property {\cite[Definition~6.4]{ozawa2015limit}}]
\label{def:common-ptp}
We say that a sequence of mm-spaces $(X_n,d_n,\mu_n)$, denoted simply by $X_n$, has the \emph{phase transition property} if there is a positive sequence $(s_n)$ such that the following assertions hold for every positive sequence $(\alpha_n)$.
\begin{enumerate}[label=\textup{(\roman*)}]
  \item If $\alpha_n/s_n\to0$, then $(X_n,\alpha_nd_n,\mu_n)$ is a L\'evy family.
  \item If $\alpha_n/s_n\to+\infty$, then $(X_n,\alpha_nd_n,\mu_n)$ infinitely dissipates.
\end{enumerate}
\end{defn}

\begin{defn}[Critical scale order {\cite[Definition~6.4]{ozawa2015limit}}]
\label{def:common-critical-scale-order}
Such a sequence $(s_n)$ is called a sequence of \emph{critical scale order}.
\end{defn}

\section{Railway-pyramid regime}
\label{sec:star-phase}

The radial limit determines the measure on each branch. Distinct branches are realized by spherical blocks separated at scale $n^{-1/2}$, while the reverse inclusion comes from approximating the angular profiles of $(N,R)$-measurements by finitely many branches.

The radial coordinate $r(x)\coloneqq\dH(0,x)$ of $\BetaBall$ admits a gamma-variable realization. For $k>0$, let $\Gamma(k,1)$ denote the distribution on $(0,+\infty)$ with density $u^{k-1}e^{-u}/\Gamma(k)$. For $n\ge1$ and $\beta>0$, let $\mathbf U_n\sim\Gamma(n/2,1)$ and $\mathbf V_{n,\beta}\sim\Gamma(\beta,1)$ be independent, and put
\[\mathbf R_{n,\beta}\coloneqq \operatorname{arcosh}\left(1+\frac{2\mathbf U_n}{\mathbf V_{n,\beta}}\right).\]

% v2: def:gamma-beta; lem:beta-gamma; lem:exact-radial-shell-inputs
\begin{lem}[Radial distribution and shell inputs]
\label{lem:common-radial-law}
For $a,b>0$, let $\operatorname{Beta}(a,b)$ be the distribution on $(0,1)$ with density
\[\frac{\Gamma(a+b)}{\Gamma(a)\Gamma(b)}t^{a-1}(1-t)^{b-1}.\]
For $n\ge1$ and $\beta>0$, let $\sigma_{n-1}$ be the normalized surface measure on $S^{n-1}$, let $\Theta_n\sim\sigma_{n-1}$ be independent of $\mathbf U_n$ and $\mathbf V_{n,\beta}$, and put
\[\mathbf X_{n,\beta}\coloneqq \Theta_n\tanh\frac{\mathbf R_{n,\beta}}2.\]
Then $\mathbf X_{n,\beta}$ has distribution $\mu_{n,\beta}$, and the distribution of $\mathbf R_{n,\beta}$ is $r_*\mu_{n,\beta}$. Moreover,
\[d(r_*\mu_{n,\beta})(u)=\frac{\Gamma(n/2+\beta)}{\Gamma(n/2)\Gamma(\beta)}\frac{\tanh^{n-1}(u/2)}{\cosh^{2\beta}(u/2)}\,du\qquad(u>0).\]
\end{lem}
\begin{proof}
Let $U$ and $V$ be independent random variables with distributions $\Gamma(b,1)$ and $\Gamma(a,1)$, respectively, and put $w=U+V$ and $t=V/(U+V)$. Using $u,v$ as the corresponding density variables, make the change of variables $(u,v)=(w(1-t),wt)$. The absolute value of the Jacobian determinant is $w$, and the joint density factors as
\[\frac{w^{a+b-1}e^{-w}}{\Gamma(a+b)}\frac{\Gamma(a+b)}{\Gamma(a)\Gamma(b)}t^{a-1}(1-t)^{b-1}.\]
The displayed factorization shows that $w$ and $t$ are independent and have distributions $\Gamma(a+b,1)$ and $\operatorname{Beta}(a,b)$, respectively.

In Euclidean polar coordinates $x=s\theta$, the measure $\mu_{n,\beta}$ is proportional to $s^{n-1}(1-s^2)^{\beta-1}\,ds\,d\sigma_{n-1}(\theta)$. Thus $t=1-s^2$ has distribution $\operatorname{Beta}(\beta,n/2)$ and is independent of the direction. Apply the preceding factorization with $a=\beta$ and $b=n/2$. The half-angle identity
\[1-\tanh^2\frac{\mathbf R_{n,\beta}}2=\frac{\mathbf V_{n,\beta}}{\mathbf U_n+\mathbf V_{n,\beta}}\]
shows that $\mathbf X_{n,\beta}$ has distribution $\mu_{n,\beta}$ and that its hyperbolic radial coordinate is $\mathbf R_{n,\beta}$. Finally, under the change of variables $u=2\operatorname{artanh}s$,
\[1-s^2=\cosh^{-2}(u/2),\qquad |d(1-s^2)|=\frac{\tanh(u/2)}{\cosh^2(u/2)}\,du.\]
Substituting these identities into the beta density gives the stated radial density. This completes the proof.
\end{proof}

\begin{remark}[Relation with the generalized Cauchy measure]
\label{rem:cauchy-coordinates}
Consider the map
\[\Psi\colon\B^n\longrightarrow\R^n,\qquad\Psi(x)\coloneqq\frac{x}{\sqrt{1-|x|^2}}.\]
Its inverse is $\Psi^{-1}(y)=y/\sqrt{1+|y|^2}$, and
\[1-|\Psi^{-1}(y)|^2=\frac1{1+|y|^2},\qquad |\det D\Psi^{-1}(y)|=(1+|y|^2)^{-n/2-1}.\]
Therefore,
\[d(\Psi_*\mu_{n,\beta})(y)=C_{n,\beta}(1+|y|^2)^{-n/2-\beta}\,dy.\]
In the convention of \cite[(1.3)]{esaki-kazukawa-mitsuishi2024cones}, the generalized Cauchy measure with exponent $\alpha>0$ has density
\[\frac{\Gamma((n+\alpha)/2)}{\pi^{n/2}\Gamma(\alpha/2)}(1+|y|^2)^{-(n+\alpha)/2}\,dy.\]
Setting $\alpha=2\beta$ gives the displayed pushforward density. The identification is only at the level of measures. This paper retains the Poincar\'e hyperbolic distance and therefore considers a different mm-space from the generalized Cauchy space with Euclidean distance. Indeed, the metric on the shell of hyperbolic radius $R>0$ is the metric transform $F_{\sinh R/\sqrt n}$ of the high-dimensional sphere. This angular amplification produces a different critical-limit classification. We follow \cite{kazukawa2022metric} for the terminology of metric transforms.
\end{remark}

For $a,b>0$, let $\operatorname{BetaPrime}(a,b)$ be the probability measure on $(0,+\infty)$ with density
\[\frac{\Gamma(a+b)}{\Gamma(a)\Gamma(b)}w^{a-1}(1+w)^{-a-b}\,dw.\]

\begin{cor}[Beta-prime distribution in Cauchy coordinates]
\label{cor:cauchy-betaprime}
For the map $\Psi$ in \Cref{rem:cauchy-coordinates},
\[|\Psi(\mathbf X_{n,\beta})|^2=\frac{\mathbf U_n}{\mathbf V_{n,\beta}}\sim\operatorname{BetaPrime}(n/2,\beta).\]
Moreover,
\[r(\mathbf X_{n,\beta})=\mathbf R_{n,\beta}=2\operatorname{arsinh}|\Psi(\mathbf X_{n,\beta})|.\]
\end{cor}
\begin{proof}
The gamma-variable realization in \Cref{lem:common-radial-law} gives
\[|\Psi(\mathbf X_{n,\beta})|^2=\frac{|\mathbf X_{n,\beta}|^2}{1-|\mathbf X_{n,\beta}|^2}=\frac{\mathbf U_n}{\mathbf V_{n,\beta}}.\]
For independent random variables $U\sim\Gamma(a,1)$ and $V\sim\Gamma(b,1)$, the change of variables $u=wv$ gives
\[
\begin{aligned}
\mathbb P(U/V\in dw)&=\frac{w^{a-1}}{\Gamma(a)\Gamma(b)}\int_0^{+\infty}v^{a+b-1}e^{-(1+w)v}\,dv\,dw\\
&=\frac{\Gamma(a+b)}{\Gamma(a)\Gamma(b)}w^{a-1}(1+w)^{-a-b}\,dw.
\end{aligned}
\]
Substituting $a=n/2$ and $b=\beta$ proves the first assertion. For every $x\in\B^n$,
\[\sinh(\operatorname{artanh}|x|)=\frac{|x|}{\sqrt{1-|x|^2}}=|\Psi(x)|.\]
Thus $r(x)=2\operatorname{artanh}|x|=2\operatorname{arsinh}|\Psi(x)|$. Taking $x=\mathbf X_{n,\beta}$ proves the last identity. This completes the proof.
\end{proof}

% v2: prop:small-beta-radial
% v2: lem:radius-log-form
\begin{lem}
\label{lem:common-radius-log}
For $t\in(0,1)$, put
\[\varrho(t)\coloneqq \operatorname{arcosh}(2/t-1),\qquad \mathfrak r(t)\coloneqq 2\log(1+\sqrt{1-t}).\]
Then
\[\varrho(t)=-\log t+\mathfrak r(t),\qquad0<\mathfrak r(t)<\log4.\]
Consequently,
\[\mathbf R_{n,\beta_n}=\log(\mathbf U_n+\mathbf V_{n,\beta_n})-\log\mathbf V_{n,\beta_n}+\mathfrak r\left(\frac{\mathbf V_{n,\beta_n}}{\mathbf U_n+\mathbf V_{n,\beta_n}}\right).\]
\end{lem}
\begin{proof}
Substitute $x=2/t-1$ into $\operatorname{arcosh}x=\log(x+\sqrt{x^2-1})$ and collect the square to obtain
\[\varrho(t)=-\log t+2\log(1+\sqrt{1-t}).\]
The bound on $\mathfrak r(t)$ follows from $0<\sqrt{1-t}<1$. Finally, substitute $t=\mathbf V_{n,\beta_n}/(\mathbf U_n+\mathbf V_{n,\beta_n})$. This completes the proof.
\end{proof}

% v2: lem:small-gamma
\begin{lem}
\label{lem:common-small-gamma}
If $V_\beta\sim\Gamma(\beta,1)$, then, as $\beta\downarrow0$,
\[-\beta\log V_\beta\overset{\mathrm d}{\longrightarrow}\operatorname{Exp}(1).\]
\end{lem}
\begin{proof}
For every $t\in\R$, the characteristic function is
\[\mathbb E[e^{it(-\beta\log V_\beta)}]=\frac{\Gamma(\beta(1-it))}{\Gamma(\beta)}=\frac{\Gamma(1+\beta(1-it))}{(1-it)\Gamma(1+\beta)}\]
and converges to $(1-it)^{-1}$ by continuity of the Gamma function at $1$. This is the characteristic function of $\operatorname{Exp}(1)$, so L\'evy's continuity theorem gives the result. This completes the proof.
\end{proof}

\begin{prop}[Radial limit for small $\beta_n$]
\label{prop:star-radial}
If $\beta_n\to0$ and $\beta_n\log n\to\lambda\in[0,+\infty)$ as $n\to\infty$, then
\[\beta_n\mathbf R_{n,\beta_n}\overset{\mathrm d}{\longrightarrow}\lambda+\operatorname{Exp}(1).\]
\end{prop}
\begin{proof}
The identity in \Cref{lem:common-radius-log} gives
\[\beta_n\mathbf R_{n,\beta_n}=\beta_n\log\frac{\mathbf U_n+\mathbf V_{n,\beta_n}}{n/2}+\beta_n\log\frac n2-\beta_n\log\mathbf V_{n,\beta_n}+\beta_n\mathfrak r\left(\frac{\mathbf V_{n,\beta_n}}{\mathbf U_n+\mathbf V_{n,\beta_n}}\right).\]
We have $\mathbb E[2\mathbf U_n/n]=1$ and $\operatorname{Var}(2\mathbf U_n/n)=2/n$. Markov's inequality also gives, for every $\varepsilon>0$,
\[\mathbb P(2\mathbf V_{n,\beta_n}/n>\varepsilon)\le\frac{2\beta_n}{n\varepsilon}\to0.\]
It follows that $(\mathbf U_n+\mathbf V_{n,\beta_n})/(n/2)\to1$ in probability, and the first term converges to $0$ in probability. The second term is $\beta_n\log n-\beta_n\log2\to\lambda$, while the absolute value of the last term is at most $\beta_n\log4\to0$. The third term converges in distribution to $\operatorname{Exp}(1)$ by \Cref{lem:common-small-gamma}. Slutsky's theorem now gives the conclusion. This completes the proof.
\end{proof}

% v2: lem:small-beta-scale-identification
% v2: lem:lebesgue-scale-algebra
\begin{lem}
\label{lem:common-scale-algebra}
For every positive sequence $(\beta_n)$, as $n\to\infty$,
\[\frac{\beta_n}{\sqrt n}\to0\quad\Longleftrightarrow\quad\mathcal A_{n,\beta_n}\to+\infty.\]
Under these equivalent conditions, $\mathcal L_{n,\beta_n}\to+\infty$.
\end{lem}
\begin{proof}
Both terms in the identity $\mathcal A_{n,\beta_n}^2=n/\beta_n^2+2/\beta_n$ are nonnegative. Thus the condition on the left implies the divergence on the right. Conversely, if $\beta_n/\sqrt n$ is bounded below by a positive constant along a subsequence, then $n/\beta_n^2$ is bounded there and $\beta_n\to+\infty$, so $2/\beta_n\to0$. Thus $\mathcal A_{n,\beta_n}$ is bounded along that subsequence. This proves the converse by contraposition. The last assertion follows from $\mathcal L_{n,\beta_n}=2\log\mathcal A_{n,\beta_n}$. This completes the proof.
\end{proof}

\begin{lem}
\label{lem:star-scale-identification}
For every positive sequence $(\beta_n)$, the following equivalences hold as $n\to\infty$.
\begin{enumerate}[label=\textup{(\roman*)}]
  \item For every $\lambda\in[0,+\infty)$,
  \[\mathcal A_{n,\beta_n}\to+\infty,\quad\beta_n\mathcal L_{n,\beta_n}\to\lambda\quad\Longleftrightarrow\quad\beta_n\to0,\quad\beta_n\log n\to\lambda.\]
  \item
  \[\mathcal A_{n,\beta_n}\to+\infty,\quad\beta_n\mathcal L_{n,\beta_n}\to+\infty\quad\Longleftrightarrow\quad\frac{\beta_n}{\sqrt n}\to0,\quad\beta_n\log n\to+\infty.\]
  \item For every $a\in(0,+\infty)$,
  \[\mathcal A_{n,\beta_n}\to a\quad\Longleftrightarrow\quad\frac{\sqrt n}{\beta_n}\to a.\]
  \item
  \[\mathcal A_{n,\beta_n}\to0\quad\Longleftrightarrow\quad\frac{\sqrt n}{\beta_n}\to0.\]
\end{enumerate}
\end{lem}
\begin{proof}
The identities
\[\mathcal A_{n,\beta_n}^2=\frac n{\beta_n^2}+\frac2{\beta_n}\]
and
\[\mathcal L_{n,\beta_n}=\log n-2\log\beta_n+\log\left(1+\frac{2\beta_n}{n}\right)\]
will be used throughout. The equivalence between $\beta_n/\sqrt n\to0$ and $\mathcal A_{n,\beta_n}\to+\infty$, as well as the resulting divergence of $\mathcal L_{n,\beta_n}$, is given by \Cref{lem:common-scale-algebra}.

For the forward implication in (i), the first equivalence gives $\mathcal L_{n,\beta_n}\to+\infty$. The convergence of $\beta_n\mathcal L_{n,\beta_n}$ to a finite value then implies $\beta_n\to0$. After multiplying the second identity by $\beta_n$, use
\[0\le\beta_n\log\left(1+\frac{2\beta_n}{n}\right)\le\frac{2\beta_n^2}{n}\to0\]
and $\beta_n\log\beta_n\to0$ to obtain $\beta_n\log n\to\lambda$. Conversely, suppose that $\beta_n\to0$ and $\beta_n\log n\to\lambda$. Since $\mathcal A_{n,\beta_n}^2\ge2/\beta_n$, we have $\mathcal A_{n,\beta_n}\to+\infty$. The second identity and the same two estimates give $\beta_n\mathcal L_{n,\beta_n}\to\lambda$.

For the forward implication in (ii), the first equivalence gives $\beta_n/\sqrt n\to0$. Rearranging the second identity and using $-2x\log x\le2/e$ for every $x>0$, we obtain
\[\beta_n\log n\ge\beta_n\mathcal L_{n,\beta_n}-\frac2e-\frac{2\beta_n^2}{n}\to+\infty.\]
Conversely, suppose that $\beta_n/\sqrt n\to0$ and $\beta_n\log n\to+\infty$. The first equivalence gives $\mathcal A_{n,\beta_n}\to+\infty$. At every index for which $\beta_n\le1$, the second identity gives $\beta_n\mathcal L_{n,\beta_n}\ge\beta_n\log n$. At every index for which $\beta_n>1$, it gives
\[\beta_n\mathcal L_{n,\beta_n}\ge2\beta_n\log\frac{\sqrt n}{\beta_n}\to+\infty.\]
The two bounds prove $\beta_n\mathcal L_{n,\beta_n}\to+\infty$ along the whole sequence.

Suppose next that $\mathcal A_{n,\beta_n}\to a\in(0,+\infty)$. The first identity shows that $n/\beta_n^2$ is bounded, and hence $\beta_n\to+\infty$. It then yields $n/\beta_n^2\to a^2$, which is equivalent to $\sqrt n/\beta_n\to a$. Conversely, if $\sqrt n/\beta_n\to a$, then $\beta_n\to+\infty$, and the first identity gives $\mathcal A_{n,\beta_n}^2\to a^2$. This proves (iii).

Finally, the first identity shows directly that $\mathcal A_{n,\beta_n}\to0$ implies $\sqrt n/\beta_n\to0$. In the converse direction, $\sqrt n/\beta_n\to0$ implies $\beta_n\to+\infty$, so both terms on the right-hand side of the first identity tend to $0$. This proves (iv) and completes the proof.
\end{proof}

% v2: lem:cross-branch-asymptotic
% v2: lem:additive-error-measurement-inclusion
The $(N,R)$-measurements transfer additive-error domination to weak pyramid limits. For $N\ge1$ and $R\ge0$, put
\[B_R^N\coloneqq\{z\in\R^N\mid\|z\|_\infty\le R\}.\]

% v2: def:measurement; prop:measurement-facts
\begin{defn}[$N$-measurements {\cite[Definition~2.25]{ozawa2015limit}}]
For $N\ge1$, the \emph{$N$-measurement} of an mm-space $X$ is the set $\mathcal M(X;N)$ of all push-forward measures $\Psi_*\mu_X$ induced by $1$-Lipschitz maps $\Psi\colon X\to(\R^N,\|\cdot\|_\infty)$.
\end{defn}

\begin{defn}[$(N,R)$-measurements {\cite[Definition~2.26]{ozawa2015limit}}]
For $R>0$, put
\[\mathcal M(X;N,R)\coloneqq \{\nu\in\mathcal M(X;N)\mid\operatorname{supp}\nu\subset B_R^N\}\]
and call this the \emph{$(N,R)$-measurement} of $X$. For a pyramid $\mathcal P$, put
\[\mathcal M(\mathcal P;N,R)\coloneqq \{\nu\mid(\operatorname{supp}\nu,\|\cdot\|_\infty,\nu)\in\mathcal P,\ \operatorname{supp}\nu\subset B_R^N\}.\]
\end{defn}

\begin{defn}[Prokhorov distance {\cite[Definition~2.12]{ozawa2015limit}}]
\label{def:common-measurement}
For Borel probability measures $\mu$ and $\nu$ on $\R^N$, their \emph{Prokhorov distance} is
\[
d_P(\mu,\nu)\coloneqq \inf\left\{\varepsilon>0\ \middle|\
\mu(A)\le\nu\bigl(\mathrm B(A,\varepsilon;\|\cdot\|_\infty)\bigr)+\varepsilon
\text{ for every Borel set }A\right\}.
\]
\label{def:common-measurement-convergence}
Weak convergence $\mathcal P_n\to\mathcal P$ is equivalent to
\[\dHaus{d_P}\bigl(\mathcal M(\mathcal P_n;N,R),\mathcal M(\mathcal P;N,R)\bigr)\to0\]
for every $N\ge1$ and $R>0$ \cite[Lemma~3.6]{ozawa2015limit}.
For pyramids $\mathcal P,\mathcal Q\in\Pi$, define
\[d_\Pi(\mathcal P,\mathcal Q)\coloneqq \sum_{N=1}^{\infty}\frac{1}{2N2^N}\dHaus{d_P}\bigl(\mathcal M(\mathcal P;N,N),\mathcal M(\mathcal Q;N,N)\bigr).\]
This paper uses $d_\Pi$ for the metric denoted by $\rho$ in \cite[Definition~4.5]{shioya2022sugaku}. It metrizes weak convergence and makes $\Pi$ a compact metric space \cite[Theorem~4.6]{shioya2022sugaku}.
\end{defn}

For $R>0$, let $\mathcal M(N,R)$ be the set of Borel probability measures on $\R^N$ whose supports are contained in $B_R^N$. The sets $\mathcal M(N,R)$, $\mathcal M(X;N,R)$, and $\mathcal M(\mathcal P;N,R)$ are compact with respect to the Prokhorov distance. Moreover, $\mathcal M(\Pyr(X);N,R)=\mathcal M(X;N,R)$ \cite[Definitions~5.37 and~7.4]{shioya2016mmg}. More generally, if a coupling $\xi$ of Borel probability measures $\mu$ and $\nu$ on a metric space $(Z,d)$ satisfies $\xi(\{(z,z')\mid d(z,z')>\varepsilon\})\le\varepsilon$, then $d_P(\mu,\nu)\le\varepsilon$.

\begin{lem}[Additive-error domination and $(N,R)$-measurements]
\label{lem:common-additive-measurement}
Let $0\le\varepsilon<1$ and suppose that $X\succ_\varepsilon Y$. Then, for every $N\ge1$ and $R>0$,
\[\mathcal M(Y;N,R)\subset\mathrm B\bigl(\mathcal M(X;N,R),\varepsilon;d_P\bigr).\]
\end{lem}
\begin{proof}
Choose a coupling $\pi\in\mathcal T(\mu_X,\mu_Y)$ and a closed set $S\subset X\times Y$ witnessing $X\succ_\varepsilon Y$, so that $\operatorname{dis}_\succ S\le\varepsilon$ and $\pi(S)\ge1-\varepsilon$. Fix $\nu=\Psi_*\mu_Y\in\mathcal M(Y;N,R)$, where $\Psi=(\Psi_1,\ldots,\Psi_N)\colon Y\to B_R^N$ is $1$-Lipschitz. Since $S\ne\emptyset$, define
\[\widetilde\Psi_j(x)\coloneqq \inf_{(u,v)\in S}\{\Psi_j(v)+d_X(x,u)\}\]
for $j=1,\ldots,N$. This infimum is finite, and the triangle inequality shows that each $\widetilde\Psi_j$ is $1$-Lipschitz. If $(x,y)\in S$, then $\operatorname{dis}_\succ S\le\varepsilon$ and the Lipschitz property of $\Psi_j$ give
\[\Psi_j(y)-\varepsilon\le\widetilde\Psi_j(x)\le\Psi_j(y)\]
for every $j$. Let $\Phi\colon X\to B_R^N$ be obtained from $\widetilde\Psi=(\widetilde\Psi_1,\ldots,\widetilde\Psi_N)$ by clipping each coordinate to $[-R,R]$. Then $\Phi$ is $1$-Lipschitz and satisfies $\|\Phi(x)-\Psi(y)\|_\infty\le\varepsilon$ on $S$.

Pushing $\pi$ forward by $(\Phi,\Psi)$ gives a coupling of $\Phi_*\mu_X$ and $\nu$ under which the set where the distance exceeds $\varepsilon$ has mass at most $\varepsilon$. The coupling implication for the Prokhorov distance yields $d_P(\Phi_*\mu_X,\nu)\le\varepsilon$. Since $\Phi_*\mu_X\in\mathcal M(X;N,R)$ and $\nu$ was arbitrary, the assertion follows. This completes the proof.
\end{proof}

% v2: cor:box-measurement-hausdorff-bound; cor:box-close-same-pyramid-limits
\begin{cor}[Passing from $\Box$-approximation to pyramid limits]
\label{cor:common-box-limit}
For all mm-spaces $X,Y$, every $N\ge1$, and every $R>0$,
\[\dHaus{d_P}\bigl(\mathcal M(X;N,R),\mathcal M(Y;N,R)\bigr)\le\Box(X,Y)\]
holds. Consequently, if sequences of mm-spaces $X_n,Y_n$ satisfy $\Box(X_n,Y_n)\to0$ as $n\to\infty$, then $\Pyr(X_n)$ and $\Pyr(Y_n)$ have the same weak subsequential limits. In particular, if one sequence converges weakly, then the other converges weakly to the same pyramid.
\end{cor}
\begin{proof}
If $\Box(X,Y)=1$, then the first estimate follows from $d_P\le1$. Suppose that $\Box(X,Y)<\delta<1$. A relation whose absolute metric distortion is controlled in the definition of $\Box$ gives both $X\succ_\delta Y$ and $Y\succ_\delta X$. Applying \Cref{lem:common-additive-measurement} in both directions shows that each $(N,R)$-measurement lies in the closed $d_P$-neighborhood of radius $\delta$ of the other. Thus their Hausdorff distance is at most $\delta$. Letting $\delta\downarrow\Box(X,Y)$ proves the first estimate.

Now suppose that $\Box(X_n,Y_n)\to0$ and that a subsequence $\Pyr(X_{n_k})$ converges weakly to $\mathcal P$. By \Cref{def:common-measurement-convergence}, for every $N,R$, the set $\mathcal M(X_{n_k};N,R)$ converges to $\mathcal M(\mathcal P;N,R)$ in Hausdorff distance. The first estimate and the triangle inequality for the Hausdorff distance give the same limit for $\mathcal M(Y_{n_k};N,R)$. Another application of \Cref{def:common-measurement-convergence} yields $\Pyr(Y_{n_k})\to\mathcal P$. Interchanging $X_n$ and $Y_n$ proves that the weak subsequential limits coincide. This completes the proof.
\end{proof}

% v2: eq:polar-distance-general; sec:hyperbolic-ball; sphere limit closes via \cite[Theorem~7.40]{shioya2016mmg}
\begin{remark}[Hyperbolic polar coordinates]
\label{set:common-hyperbolic-polar}
For $r>0$ and $\theta\in S^{n-1}$, write $(r,\theta)\coloneqq\tanh(r/2)\theta$. In hyperbolic polar coordinates,
\begin{equation}\label{eq:common-polar-distance}
\cosh\dH((r,\theta),(s,\eta))=\cosh(r-s)+\frac12\sinh r\sinh s\,|\theta-\eta|^2. \tag{H1}
\end{equation}
Put
\[\mathsf S_n\coloneqq (S^{n-1},\sqrt n|\theta-\eta|,\sigma_{n-1}).\]
Then $\Pyr(\mathsf S_n)\to\PGauss$ weakly as $n\to\infty$.
\end{remark}
\begin{proof}
The hyperbolic law of cosines and $\langle\theta,\eta\rangle=1-|\theta-\eta|^2/2$ give
\[\cosh\dH((r,\theta),(s,\eta))=\cosh r\cosh s-\sinh r\sinh s+\frac12\sinh r\sinh s|\theta-\eta|^2\]
and the first two terms equal $\cosh(r-s)$. This proves \eqref{eq:common-polar-distance}.

By \cite[Theorem~7.40]{shioya2016mmg}, the pyramids associated with the spheres whose chord metrics are multiplied by $\sqrt{n-1}$ converge weakly to $\PGauss$. On the same sphere, the difference between the two metrics is at most
\[\sup_{\theta,\eta\in S^{n-1}}(\sqrt n-\sqrt{n-1})|\theta-\eta|\le2(\sqrt n-\sqrt{n-1})\to0\]
as $n\to\infty$. The diagonal coupling gives the corresponding $\Box$-estimate, and \Cref{cor:common-box-limit} yields $\Pyr(\mathsf S_n)\to\PGauss$. This completes the proof.
\end{proof}

\begin{lem}[Cross-branch distance asymptotics]
\label{lem:star-cross-branch-asymptotic}
Assume that $\beta_n\to0$ and $\beta_n\log n\to\lambda$ as $n\to\infty$, and fix $\delta>0$ and $M>\lambda+\delta$.
\begin{enumerate}[label=\textup{(\roman*)}]
  \item For every $c>0$ and $\varepsilon>0$, if $n$ is sufficiently large, $\beta_n\rho,\beta_n\sigma\in[\lambda+\delta,M]$, and $|\theta-\eta|\ge c/\sqrt n$, then
  \[\beta_n\dH((\rho,\theta),(\sigma,\eta))\ge\beta_n\rho+\beta_n\sigma-\lambda-\varepsilon.\]
  \item For every $C>0$ and $\varepsilon>0$, if $n$ is sufficiently large, $\beta_n\rho,\beta_n\sigma\in[\lambda+\delta,M]$, and $|\theta-\eta|\le C/\sqrt n$, then
  \[\beta_n\dH((\rho,\theta),(\sigma,\eta))\le\beta_n\rho+\beta_n\sigma-\lambda+\varepsilon.\]
\end{enumerate}
Both estimates are uniform over $\rho,\sigma,\theta,\eta$ in the displayed ranges.
\end{lem}
\begin{proof}
On the stated window, $\rho,\sigma\ge(\lambda+\delta)/\beta_n\to+\infty$. By \eqref{eq:common-polar-distance} and the estimate $\sinh u\ge e^u/4$ for sufficiently large $u$, the angular condition in (i) gives
\[\cosh\dH((\rho,\theta),(\sigma,\eta))\ge\frac12\sinh\rho\sinh\sigma|\theta-\eta|^2\ge\frac{c^2}{32n}e^{\rho+\sigma}.\]
Using $u\ge\log\cosh u$, we obtain
\[\beta_n\dH((\rho,\theta),(\sigma,\eta))\ge\beta_n(\rho+\sigma)-\beta_n\log n+\beta_n\log(c^2/32).\]
Since $\beta_n\log n\to\lambda$ and $\beta_n\to0$, this is at least the right-hand side of (i) for all sufficiently large $n$.

For (ii),
\[\log n-2\min\{\rho,\sigma\}\le\frac{\beta_n\log n-2(\lambda+\delta)}{\beta_n}\to-\infty.\]
For all sufficiently large $n$, this implies $n e^{-2\min\{\rho,\sigma\}}\le1$. Using this inequality, we obtain $\cosh(\rho-\sigma)\le e^{|\rho-\sigma|}\le e^{\rho+\sigma}/n$. Since $\sinh\rho\sinh\sigma\le e^{\rho+\sigma}/4$, \eqref{eq:common-polar-distance} gives
\[\cosh\dH((\rho,\theta),(\sigma,\eta))\le\left(1+\frac{C^2}{8}\right)\frac{e^{\rho+\sigma}}n.\]
Applying $\operatorname{arcosh}x\le\log(2x)$ for $x\ge1$, we obtain
\[\beta_n\dH((\rho,\theta),(\sigma,\eta))\le\beta_n(\rho+\sigma)-\beta_n\log n+\beta_n\log(2+C^2/4),\]
and the same two limits prove (ii). All estimates are uniform on the stated window. This completes the proof.
\end{proof}

% v2: lem:partial-lipschitz-domination
\begin{lem}
\label{lem:star-partial-domination}
Let $(X,d_X,\mu_X)$, denoted simply by $X$, be an mm-space, let $E\subset X$ be Borel with $\mu_X(E)=m>0$, and let $Y$ be a compact metric space. Suppose that a Borel map $f\colon E\to Y$ satisfies, for some $\eta\ge0$,
\[d_Y(f(x),f(x'))\le d_X(x,x')+\eta\qquad(x,x'\in E)\]
and put $\nu_E\coloneqq f_*(\mu_X|_E/m)$. Then
\[X\succ_{\max\{1-m,\eta\}}(\operatorname{supp}\nu_E,d_Y,\nu_E).\]
\end{lem}
\begin{proof}
Put $E_0\coloneqq E\cap f^{-1}(\operatorname{supp}\nu_E)$. The conditional measure $\mu_X|_E/m$ gives full mass to $E_0$. Let $S$ be the closure in $X\times\operatorname{supp}\nu_E$ of the graph of $f|_{E_0}$. The assumed inequality passes to this closure, so $\operatorname{dis}_\succ S\le\eta$.

Put
\[\pi\coloneqq (\operatorname{id},f)_*(\mu_X|_E)+(\mu_X|_{X\setminus E})\otimes\nu_E.\]
This is a coupling of $\mu_X$ and $\nu_E$ satisfying $\pi(S)\ge m$. Therefore, by \Cref{def:common-box-concentration}, the set $S$ and the coupling $\pi$ give additive-error domination with error $\max\{1-m,\eta\}$.
	This completes the proof.
\end{proof}

% v2: lem:additive-error-persistence
% v2: lem:additive-error-persistence closes this via \cite[Lemma~6.17]{shioya2016mmg}
\begin{lem}[Bounded-coordinate density {\cite[Lemma~6.17]{shioya2016mmg}}]
\label{lem:common-bounded-coordinate-density}
For every pyramid $\mathcal P$ and every $X\in\mathcal P$, there are a strictly increasing sequence $(N_j)$ and measures $\nu_j\in\mathcal M(\mathcal P;N_j,N_j)$ such that
\[\Box\bigl(X,(\operatorname{supp}\nu_j,\|\cdot\|_\infty,\nu_j)\bigr)\to0\]
as $j\to\infty$. In particular, for every mm-space $X$, the measures can be chosen so that $\nu_j\in\mathcal M(X;N_j,N_j)$.
\end{lem}

\begin{lem}
\label{lem:star-additive-persistence}
Suppose that, as $n\to\infty$, $d_\Pi(\Pyr(X_n),\mathcal P)\to0$, $X_n\succ_{\varepsilon_n}Y_n$, $\varepsilon_n\to0$, and $\Box(Y_n,Y)\to0$. Then $Y\in\mathcal P$.
\end{lem}
\begin{proof}
Fix $N\ge1$ and $R>0$. By \Cref{lem:common-additive-measurement}, for all sufficiently large $n$,
\[\mathcal M(Y_n;N,R)\subset\mathrm B\bigl(\mathcal M(X_n;N,R),\varepsilon_n;d_P\bigr).\]
The box-limit statement in \Cref{cor:common-box-limit} shows that $\mathcal M(Y_n;N,R)$ converges to $\mathcal M(Y;N,R)$ in Hausdorff distance. The definition in \Cref{def:common-measurement-convergence} likewise gives convergence of $\mathcal M(X_n;N,R)$ to $\mathcal M(\mathcal P;N,R)$. Approximate any $\nu\in\mathcal M(Y;N,R)$ by measures $\nu_n\in\mathcal M(Y_n;N,R)$ and pass to the limit through the displayed inclusion. This gives
\[\mathcal M(Y;N,R)\subset\mathcal M(\mathcal P;N,R)\]
for every $N\ge1$ and $R>0$.

The bounded-coordinate density result in \Cref{lem:common-bounded-coordinate-density} provides $N_j\uparrow+\infty$ and $\nu_j\in\mathcal M(Y;N_j,N_j)$ such that the associated mm-spaces converge to $Y$ in $\Box$-distance. The displayed inclusion places every approximating space in $\mathcal P$. Since $\mathcal P$ is $\Box$-closed, $Y\in\mathcal P$. This completes the proof.
\end{proof}

% v2: lem:spherical-two-set-distance; closes via \cite[Proposition~2.38 and Example~7.36]{shioya2016mmg}
\begin{lem}
\label{lem:common-spherical-two-set}
Let $n\ge2$, and let Borel sets $A,B\subset S^{n-1}$ satisfy $a\coloneqq\sigma_{n-1}(A)>0$ and $b\coloneqq\sigma_{n-1}(B)>0$. Then, for the chord distance,
\[\inf_{\theta\in A,\,\eta\in B}|\theta-\eta|\le\frac2{\sqrt{(n-1)\min\{a,b\}}}.\]
\end{lem}
\begin{proof}
The spectral separation estimate \cite[Proposition~2.38]{shioya2016mmg} gives, for a compact Riemannian manifold $M$,
\[\operatorname{Sep}(M;\kappa_0,\kappa_1)\le\frac2{\sqrt{\lambda_1(M)\min\{\kappa_0,\kappa_1\}}}\]
where $\lambda_1(M)$ is the first nonzero eigenvalue. For $S^{n-1}$, this eigenvalue is $n-1$ \cite[Example~7.36]{shioya2016mmg}. Thus the spherical distance between $A$ and $B$ is at most $2/\sqrt{(n-1)\min\{a,b\}}$. Since the chord distance is no greater than the spherical distance, the same upper bound holds. This completes the proof.
\end{proof}

% v2: lem:spherical-block
\begin{lem}[Separation scale on the Euclidean sphere]
\label{lem:common-spherical-block}
For fixed positive numbers $\kappa_0,\ldots,\kappa_N$ satisfying $\sum_i\kappa_i<1$, there are constants $0<c<C<+\infty$ such that, for all sufficiently large $n$,
\[\frac c{\sqrt n}\le\operatorname{Sep}\bigl((S^{n-1},|\cdot|,\sigma_{n-1});\kappa_0,\ldots,\kappa_N\bigr)\le\frac C{\sqrt n}.\]
\end{lem}
\begin{proof}
Let $Z_1,\ldots,Z_n$ be independent standard Gaussian random variables. The random vector $(Z_1,\ldots,Z_n)/(\sum_jZ_j^2)^{1/2}$ has distribution $\sigma_{n-1}$. A second-moment estimate gives $n^{-1}\sum_jZ_j^2\to1$ in probability, and therefore its first coordinate, multiplied by $\sqrt n$, converges in distribution to $\gamma^1$. Since $\sum_i\kappa_i<1$, we can choose intervals of $\gamma^1$-measure greater than $\kappa_i$ whose pairwise distances are at least some $c>0$. For all sufficiently large $n$, their inverse images are admissible sets separated in chord distance by at least $c/\sqrt n$. This proves the lower bound.

For the upper bound, put $\kappa_*\coloneqq\min_i\kappa_i$ and apply \Cref{lem:common-spherical-two-set} to any two admissible sets. For $n\ge2$,
\[\frac2{\sqrt{(n-1)\kappa_*}}\le\frac{2\sqrt{2/\kappa_*}}{\sqrt n},\]
so we may take $C=2\sqrt{2/\kappa_*}$. This completes the proof.
\end{proof}

% v2: cor:star-map-separated-blocks
\begin{cor}
\label{cor:common-separated-blocks}
Fix $m\ge2$, positive numbers $p_i$ satisfying $\sum_{i=1}^mp_i=1$, and $\alpha\in(0,1)$. Then there is $c>0$ such that, for all sufficiently large $n$, there are Borel sets $A_{1,n},\ldots,A_{m,n}\subset S^{n-1}$ satisfying
\[\sigma_{n-1}(A_{i,n})=(1-\alpha)p_i\qquad(i=1,\ldots,m)\]
and
\[|\theta-\eta|\ge\frac c{\sqrt n}\]
whenever $i\ne k$, $\theta\in A_{i,n}$, and $\eta\in A_{k,n}$.
\end{cor}
\begin{proof}
Because the masses $((1-\alpha)p_i)_{i=1}^m$ sum to $1-\alpha<1$, \Cref{lem:common-spherical-block} provides, after decreasing the constant, Borel sets of at least the prescribed masses that are pairwise separated by $c/\sqrt n$ for all sufficiently large $n$. Since $\sigma_{n-1}$ is nonatomic, trim each set to have mass exactly $(1-\alpha)p_i$ without decreasing the separation. This completes the proof.
\end{proof}

% v2: L1187 (unlabeled Strassen coupling theorem)
\begin{lem}
\label{lem:common-prokhorov-coupling}
Let $\mu,\nu$ be Borel probability measures on a complete separable metric space $(Z,d)$. If $\varepsilon>d_P(\mu,\nu)$, then there is a coupling $\xi$ of $\mu$ and $\nu$ such that
\[\xi(\{(z,z')\in Z\times Z\mid d(z,z')>\varepsilon\})\le\varepsilon.\]
In particular, if $d_P(\mu_n,\nu_n)\to0$ as $n\to\infty$, then there are $\varepsilon_n\downarrow0$ and couplings $\xi_n$ such that
\[\xi_n(\{(z,z')\mid d(z,z')>\varepsilon_n\})\le\varepsilon_n\]
for every $n$.
\end{lem}
\begin{proof}
By \cite[Theorem~1.22]{shioya2016mmg}, there are $0<\eta<\varepsilon$ and a Borel measure $m$ on $Z\times Z$ with marginals $\mu'\le\mu$ and $\nu'\le\nu$ such that
\[\operatorname{supp}m\subset\{(z,z')\mid d(z,z')\le\eta\},\qquad 1-m(Z\times Z)\le\eta.\]
Put $r\coloneqq1-m(Z\times Z)$. The two deficit measures $\mu-\mu'$ and $\nu-\nu'$ both have mass $r$. If $r>0$, define
\[\xi\coloneqq m+\frac{(\mu-\mu')\otimes(\nu-\nu')}{r},\]
and if $r=0$, put $\xi\coloneqq m$. This is a coupling of $\mu$ and $\nu$. Since $m$ is supported where the distance is at most $\eta<\varepsilon$, the $\xi$-mass of the set where the distance exceeds $\varepsilon$ is at most the total mass $r\le\eta<\varepsilon$ of the added measure.

For the second assertion, put
\[\varepsilon_n\coloneqq \sup_{k\ge n}d_P(\mu_k,\nu_k)+\frac1n\downarrow0\]
and apply the first assertion for each $n$. This completes the proof.
\end{proof}

% v2: lem:common-ambient-box
\begin{lem}
\label{lem:common-ambient-box}
Let $\mu,\nu$ be Borel probability measures on a complete separable metric space $(Z,d)$, and regard each measure as an mm-space on its support. Suppose that a coupling $\xi$ of $\mu$ and $\nu$ satisfies
\[\xi(\{(z,z')\mid d(z,z')>\delta\})\le\eta\]
for some $\delta,\eta\ge0$. Then
\[\Box\bigl((\operatorname{supp}\mu,d,\mu),(\operatorname{supp}\nu,d,\nu)\bigr)\le\max\{2\delta,\eta\}.\]
\end{lem}
\begin{proof}
Put $S\coloneqq\{(z,z')\in\operatorname{supp}\mu\times\operatorname{supp}\nu\mid d(z,z')\le\delta\}$. This set is closed and satisfies $\xi(S)\ge1-\eta$. For $(z,z'),(w,w')\in S$, the triangle inequality gives
\[|d(z,w)-d(z',w')|\le d(z,z')+d(w,w')\le2\delta.\]
The estimate $\operatorname{dis}S\le2\delta$ follows, and \Cref{def:common-box-concentration} gives the conclusion. This completes the proof.
\end{proof}

% v2: lem:star-map
\begin{lem}[Finite-branch railway mm-spaces in subsequential limits]
\label{lem:star-map}
Under the assumptions of \Cref{prop:star-radial}, fix $m\ge1$ and positive weights $p_i$ satisfying $\sum_i p_i=1$. Then $T_{m,\lambda}(p_1,\ldots,p_m)$ belongs to every weak subsequential limit of $\Pyr(\beta_n\BetaBall)$. In particular,
\[T_{1,\lambda}(1)=([\lambda,+\infty),|\cdot|,\nu_\lambda).\]
\end{lem}
\begin{proof}
Let $\mathcal P$ be an arbitrary weak subsequential limit of $\Pyr(\beta_n\BetaBall)$, and relabel the corresponding subsequence by $n$. Choose $\delta_j\downarrow0$, $M_j\to+\infty$, $\tau_j\downarrow0$, and $\alpha_j\downarrow0$ with $M_j>\lambda+\delta_j$, and put
\[I_j\coloneqq [\lambda+\delta_j,M_j],\qquad q_j\coloneqq \nu_\lambda(I_j),\]
so that $q_j\to1$ as $j\to\infty$.

The four parameters control distinct errors. Removing $[\lambda,\lambda+\delta_j)$ makes the cross-branch asymptotic uniform on the remaining radial window, while removing $(M_j,+\infty)$ makes that window compact. Their combined radial mass loss is $1-q_j$. The quantity $\alpha_j$ is the angular mass left outside the separated blocks, and $\tau_j$ is the uniform additive error in the comparison between the scaled hyperbolic distance and the railway distance across different blocks.

If $m\ge2$, then \Cref{cor:common-separated-blocks} gives, for each $j$ and all sufficiently large $n$, spherical blocks satisfying $\sigma_{n-1}(A_{i,n}^{(j)})=(1-\alpha_j)p_i$ and
\[|\theta-\eta|\ge\frac{c_j}{\sqrt n},\]
whenever $i\ne k$, $\theta\in A_{i,n}^{(j)}$, and $\eta\in A_{k,n}^{(j)}$. If $m=1$, nonatomicity gives a set $A_{1,n}^{(j)}$ of mass $1-\alpha_j$. Put
\[E_{j,n}\coloneqq \{(\rho,\theta)\mid\beta_n\rho\in I_j,\ \theta\in\textstyle\bigcup_iA_{i,n}^{(j)}\},\]
and, for $\theta\in A_{i,n}^{(j)}$, define
\[h_{j,n}(\rho,\theta)\coloneqq (\beta_n\rho,i)\in I_j\times\{1,\ldots,m\}.\]
Same-branch distances satisfy $|\beta_n\rho-\beta_n\sigma|\le\beta_n\dH((\rho,\theta),(\sigma,\eta))$. The scaling convention gives
\[d_{\beta_n\BetaBall}(x,y)=\beta_n\dH(x,y).\]
For different branches, write $d_\lambda$ for the metric of $T_{m,\lambda}(p_1,\ldots,p_m)$. Then \Cref{lem:star-cross-branch-asymptotic}(i) gives, for all sufficiently large $n$,
\[d_\lambda(h_{j,n}(x),h_{j,n}(y))\le d_{\beta_n\BetaBall}(x,y)+\tau_j\qquad(x,y\in E_{j,n}),\]

Radial--angular independence from \Cref{lem:common-radial-law} and \Cref{prop:star-radial} give
\[\mu_{n,\beta_n}(E_{j,n})\to(1-\alpha_j)q_j.\]
Since $\nu_\lambda$ assigns zero mass to the endpoints of $I_j$, the pushforward under $h_{j,n}$ of the normalized measure on $E_{j,n}$ converges weakly on the compact space $I_j\times\{1,\ldots,m\}$ to
\[q_j^{-1}\nu_\lambda|_{I_j}\otimes\sum_{i=1}^mp_i\delta_i.\]
Let $Y_{j,n}$ be the mm-space obtained by equipping the support of this pushforward measure with the metric induced from the finite-branch railway mm-space. By \Cref{lem:star-partial-domination},
\[\beta_n\BetaBall\succ_{\max\{1-\mu_{n,\beta_n}(E_{j,n}),\tau_j\}}Y_{j,n}.\]
Let $Y_j^I$ be the mm-space defined by the limiting measure on the radial window. On a compact space, weak convergence is equivalent to convergence in Prokhorov distance. Therefore, \Cref{lem:common-prokhorov-coupling,lem:common-ambient-box} imply $\Box(Y_{j,n},Y_j^I)\to0$ as $n\to\infty$. The measures of $Y_j^I$ and the full finite-branch railway mm-space can be coupled diagonally on the mass $q_j$ inside the window, so \Cref{def:common-box-concentration} gives
\[\Box\bigl(Y_j^I,T_{m,\lambda}(p_1,\ldots,p_m)\bigr)\le1-q_j.\]
For every $j$, the triangle inequality yields
\[\limsup_{n\to\infty}\Box\bigl(Y_{j,n},T_{m,\lambda}(p_1,\ldots,p_m)\bigr)\le2(1-q_j).\]

Choose a strictly increasing sequence $(n_j)$ such that
\[\left|\mu_{n_j,\beta_{n_j}}(E_{j,n_j})-(1-\alpha_j)q_j\right|\le j^{-1},\qquad \Box\bigl(Y_{j,n_j},T_{m,\lambda}(p_1,\ldots,p_m)\bigr)\le j^{-1}+2(1-q_j).\]
In this diagonal extraction, the radial truncation error $1-q_j$, the discarded angular mass $\alpha_j$, the metric error $\tau_j$, and the selection error $j^{-1}$ all tend to zero as $j\to\infty$. Hence the domination errors tend to zero, and $Y_{j,n_j}$ converges in $\Box$-distance to the required finite-branch railway mm-space. Thus \Cref{lem:star-additive-persistence} gives $T_{m,\lambda}(p_1,\ldots,p_m)\in\mathcal P$. When $m=1$, there are no cross-branch pairs, and the last identity follows from the measure and metric defined in \Cref{thm:phase-diagram-draft}. This completes the proof.
\end{proof}

% v2: def:star-pyramid + following directed-union verification
\begin{lem}
\label{lem:star-pyramid}
The family $\mathcal P_\lambda^\star$ defined in \Cref{thm:phase-diagram-draft} is a pyramid.
\end{lem}
\begin{proof}
The union in the definition is downward-closed under domination. For two generators $T_{m,\lambda}(p_1,\ldots,p_m)$ and $T_{\ell,\lambda}(q_1,\ldots,q_\ell)$, take $T_{m\ell,\lambda}$ whose branches are indexed by pairs $(i,j)$ and have weights $p_iq_j$. The branch projections
\[(s,(i,j))\longmapsto(s,i),\qquad(s,(i,j))\longmapsto(s,j)\]
are measure-preserving. Each projection either preserves the cross-branch distance $s+t-\lambda$ or decreases it to the same-branch distance $|s-t|\le s+t-\lambda$. Thus both projections are $1$-Lipschitz, and the union is nonempty, upward-directed, and downward-closed. The two stability statements for domination under $\Box$-approximation in \cite[Lemma~6.10]{shioya2016mmg} show that its $\Box$-closure remains upward-directed and downward-closed. Since it is $\Box$-closed by definition, it satisfies the three conditions in \Cref{def:common-pyramid}. This completes the proof.
\end{proof}

% v2: lem:star-rich-separation
\begin{lem}[Positive separation from finite-branch railway mm-spaces]
\label{lem:star-rich-separation}
Let $\mathcal P$ be a pyramid containing $T_{m,\lambda}(p_1,\ldots,p_m)$ for every $m\ge1$ and every choice of positive weights $p_i$. For every $N\ge1$ and positive numbers $\kappa_0,\ldots,\kappa_N$ satisfying $\sum_{i=0}^N\kappa_i<1$, there are $Y\in\mathcal P$ and $\varepsilon>0$ such that
\[\operatorname{Sep}(Y;\kappa_0+\varepsilon,\ldots,\kappa_N+\varepsilon)>0.\]
\end{lem}
\begin{proof}
Choose weights $p_i>\kappa_i$ with $\sum_i p_i=1$, and then choose $\varepsilon>0$ such that $p_i>\kappa_i+\varepsilon$ for every $i$. The tree $Y\coloneqq T_{N+1,\lambda}(p_0,\ldots,p_N)$ belongs to $\mathcal P$. The tail $A_i\coloneqq\{(s,i)\mid s\ge r_i\}$ of branch $i$ has mass
\[\mu_Y(A_i)=p_i e^{-(r_i-\lambda)}.\]
Taking $r_i>\lambda$ sufficiently close to $\lambda$ makes this mass at least $\kappa_i+\varepsilon$. Distinct tails have distance $r_i+r_j-\lambda>0$, so they give the required separation.
	This completes the proof.
\end{proof}

% v2: lem:star-pyramid-finite-obsdiam
% v2: lem:obsdiam-basic
\begin{lem}[Basic estimates for observable diameter {\cite[Proposition~2.18(2) and (3)]{shioya2016mmg}}]
\label{lem:common-obsdiam-basic}
For mm-spaces $X,Y$ and $\kappa>0$,
\[\operatorname{ObsDiam}(X;-\kappa)\le\operatorname{PartDiam}(\mu_X;1-\kappa).\]
Moreover, if $Y\prec X$, then
\[\operatorname{ObsDiam}(Y;-\kappa)\le\operatorname{ObsDiam}(X;-\kappa).\]
\end{lem}

% v2: eq:pyramid-obsdiam-limit; eq:pyramid-separation-limit
\begin{prop}[Limit formulas for weak convergence of pyramids {\cite[Theorem~3.12]{ozawa2015limit}}, {\cite[Theorem~1.1]{yokota2024obsdiam}}, {\cite[Theorem~4.9 and Corollary~4.8]{ozawa2015limit}}]
\label{prop:common-pyramid-limit-formulas}
If $\mathcal P_n\to\mathcal P$ weakly as $n\to\infty$, then, for every $\kappa>0$,
\[
\begin{aligned}
\operatorname{ObsDiam}(\mathcal P;-\kappa)&=\lim_{\varepsilon\downarrow0}\liminf_{n\to\infty}\operatorname{ObsDiam}(\mathcal P_n;-(\kappa+\varepsilon))\\
&=\lim_{\varepsilon\downarrow0}\limsup_{n\to\infty}\operatorname{ObsDiam}(\mathcal P_n;-(\kappa+\varepsilon))
\end{aligned}
,\]
and, for every choice of positive numbers $\kappa_0,\ldots,\kappa_N$,
\[
\begin{aligned}
\operatorname{Sep}(\mathcal P;\kappa_0,\ldots,\kappa_N)&=\lim_{\varepsilon\downarrow0}\liminf_{n\to\infty}\operatorname{Sep}(\mathcal P_n;\kappa_0-\varepsilon,\ldots,\kappa_N-\varepsilon)\\
&=\lim_{\varepsilon\downarrow0}\limsup_{n\to\infty}\operatorname{Sep}(\mathcal P_n;\kappa_0-\varepsilon,\ldots,\kappa_N-\varepsilon)
\end{aligned}
.\]
The limits in the second formula are restricted to $0<\varepsilon<\min_i\kappa_i$.
\end{prop}

\begin{lem}
\label{lem:star-pyramid-finite-obsdiam}
For $\lambda\ge0$ and $0<\kappa<1$,
\[\operatorname{ObsDiam}(\mathcal P_\lambda^\star;-\kappa)\le\lambda+2\log\frac2\kappa.\]
For $\kappa\ge1$, the left-hand side is zero.
\end{lem}
\begin{proof}
The assertion for $\kappa\ge1$ follows from the definition. Let $0<\kappa<1$ and choose $M>\lambda+\log(2/\kappa)$. In any finite-branch railway mm-space, the set on which the radial coordinate is at most $M$ has mass
\[\nu_\lambda([\lambda,M])=1-e^{-(M-\lambda)}>1-\kappa/2\]
and diameter at most $2M-\lambda$. By \Cref{lem:common-obsdiam-basic}, for every finite-branch railway mm-space $T$ and every $Y\prec T$,
\[\operatorname{ObsDiam}(Y;-\kappa/2)\le\operatorname{ObsDiam}(T;-\kappa/2)\le2M-\lambda.\]

Fix $X\in\mathcal P_\lambda^\star$, and choose $Y_j$ in the generating union such that $\Box(Y_j,X)\to0$. Monotonicity under domination and $Y_j\in\Pyr(Y_j)$ give $\operatorname{ObsDiam}(\Pyr(Y_j);-\kappa/2)=\operatorname{ObsDiam}(Y_j;-\kappa/2)$, and the same identity holds for $X$. By \Cref{cor:common-box-limit}, $\Pyr(Y_j)\to\Pyr(X)$ weakly. Applying \Cref{prop:common-pyramid-limit-formulas} at mass parameter $\kappa/2$ to the uniform bound above gives $\operatorname{ObsDiam}(X;-\kappa/2)\le2M-\lambda$, and hence $\operatorname{ObsDiam}(X;-\kappa)\le2M-\lambda$. Take the supremum over $X$ and let $M\downarrow\lambda+\log(2/\kappa)$ to obtain the result. This completes the proof.
\end{proof}

% v2: lem:profile-compactness-cross-constraint
\begin{lem}[Profile compactness and cross-branch constraint]
\label{lem:star-profile-compactness}
Assume that $\beta_n\to0$ and $\beta_n\log n\to\lambda\in[0,+\infty)$ as $n\to\infty$, and put $Z_n\coloneqq\beta_n\BetaBall$. Fix $N\ge1$, $R>0$, and a sequence $n_k\to+\infty$. Let $\Psi_k\colon Z_{n_k}\to B_R^N$ be $1$-Lipschitz maps satisfying
\[(\Psi_k)_*\mu_{n_k,\beta_{n_k}}\overset{\mathrm d}{\longrightarrow}\mu,\]
as $k\to\infty$. For $\delta>0$ and $M>\lambda+\delta$, put $I\coloneqq[\lambda+\delta,M]$. Equip $K\coloneqq\operatorname{Lip}_1(I,B_R^N)$ with the uniform metric
\[d_K(g,h)\coloneqq\sup_{s\in I}\|g(s)-h(s)\|_\infty,\]
and write $B_K(g,\tau)\coloneqq\{h\in K\mid d_K(g,h)<\tau\}$. Define
\[g_{k,\theta}(s)\coloneqq \Psi_k(s/\beta_{n_k},\theta),\qquad q_k\coloneqq (\theta\mapsto g_{k,\theta})_*\sigma_{n_k-1}.\]
Then, after passing to a subsequence, $q_k\Rightarrow q$ for some probability measure $q$ on $K$, and
\[\|g(s)-h(t)\|_\infty\le s+t-\lambda\qquad(g,h\in\operatorname{supp}q,\ g\ne h,\ s,t\in I).\]
Moreover, if $E(s,g)\coloneqq g(s)$, then
\[\mu^I\coloneqq E_*\bigl((\nu_\lambda|_I)\otimes q\bigr)\le\mu,\qquad(\mu-\mu^I)(B_R^N)=1-\nu_\lambda(I).\]
\end{lem}
The measure $q$ records the limiting distribution, over angular directions, of the profiles $g_{k,\theta}$ on $I$, and the inequality between $g$ and $h$ is the limiting Lipschitz constraint between profiles from distinct branches.
\begin{proof}[Proof sketch]
The Arzel\`a--Ascoli theorem gives a subsequential limit of the profile measures. The distance estimate between two spherical sets and the cross-branch asymptotic yield the constraint between distinct limiting profiles. Radial--angular independence and the radial limit then show that the pushforward measure on the window converges to $\mu^I$. The full argument is given in \Cref{sec:app-star-profile}.
	This completes the proof.
\end{proof}

% v2: lem:measurement-upper-inclusion
% v2: prop:measurement-facts(ii)
\begin{prop}
\label{prop:common-truncation}
For an mm-space $(X,d_X,\mu_X)$, or simply $X$, and $L>0$, put
\[X^{(L)}\coloneqq (X,d_X\wedge L,\mu_X).\]
Then, for every $N\ge1$ and $R>0$,
\[\mathcal M(X;N,R)=\mathcal M(X^{(2R)};N,R).\]
\end{prop}
\begin{proof}
Since $d_X\wedge2R\le d_X$, every $1$-Lipschitz map on $X^{(2R)}$ is also $1$-Lipschitz on $X$. Conversely, suppose that $\Psi_*\mu_X\in\mathcal M(X;N,R)$. Continuity of $\Psi$ and $\operatorname{supp}\mu_X=X$ give $\Psi(X)\subset\operatorname{supp}\Psi_*\mu_X\subset B_R^N$, so $\|\Psi(x)-\Psi(y)\|_\infty\le2R$. Combining this bound with the Lipschitz property of $\Psi$ on $X$ gives $\|\Psi(x)-\Psi(y)\|_\infty\le(d_X\wedge2R)(x,y)$ and proves the reverse inclusion. This completes the proof.
\end{proof}

\begin{lem}[Upper inclusion for $(N,R)$-measurements]
\label{lem:star-measurement-upper}
Assume that $\beta_n\to0$ and $\beta_n\log n\to\lambda\in[0,+\infty)$ as $n\to\infty$, and put $Z_n\coloneqq\beta_n\BetaBall$. Fix $N\ge1$, $R>0$, and a sequence $n_k\to+\infty$. If
\[\mu_k\in\mathcal M(Z_{n_k};N,R),\qquad d_P(\mu_k,\mu)\to0\]
as $k\to\infty$, then
\[\mu\in\mathcal M(\mathcal P_\lambda^\star;N,R).\]
\end{lem}
\begin{proof}[Proof sketch]
The full argument has three steps: (1) restricting to a radial window whose limiting mass exceeds $1-\varepsilon$, (2) quantizing the limiting profile measure by finitely many profiles, and (3) extending the resulting window map coordinatewise to a finite-branch railway mm-space and coupling its pushforward measure to $\mu$. The resulting approximants belong to $\mathcal M(\mathcal P_\lambda^\star;N,R)$, and their Prokhorov convergence to $\mu$, followed by $\Box$-closedness, gives the conclusion. The details are given in \Cref{sec:app-star-measurement}. This completes the proof.
\end{proof}

% v2: cor:measurements-determine-pyramid
\begin{cor}[Identification of pyramids by $(N,R)$-measurements]
\label{cor:star-measurements-determine}
If pyramids $\mathcal P,\mathcal Q$ satisfy
\[\mathcal M(\mathcal P;N,R)=\mathcal M(\mathcal Q;N,R)\]
for every $N\ge1$ and $R>0$, then $\mathcal P=\mathcal Q$.
\end{cor}
\begin{proof}
Take the constant sequence $\mathcal P_j\coloneqq\mathcal P$. By hypothesis, for every $N,R$, the sets $\mathcal M(\mathcal P_j;N,R)$ converge in Hausdorff distance to both $\mathcal M(\mathcal P;N,R)$ and $\mathcal M(\mathcal Q;N,R)$. By \Cref{def:common-measurement-convergence}, the sequence $\mathcal P_j$ converges weakly to both pyramids. Since $d_\Pi$ metrizes this convergence, uniqueness of limits in a metric space gives $\mathcal P=\mathcal Q$. This completes the proof.
\end{proof}

% v2: thm:small-beta-full
\begin{proof}[Proof of \Cref{thm:phase-diagram-draft}]
Let $\mathcal P$ be an arbitrary weak subsequential limit of $\Pyr(\beta_n\BetaBall)$. Every finite-branch railway mm-space belongs to $\mathcal P$ by \Cref{lem:star-map}. Downward and $\Box$-closedness give
\[\mathcal P_\lambda^\star\subset\mathcal P\]
for the first inclusion.

For the reverse inclusion, fix $N\ge1$ and $R>0$. By \Cref{def:common-measurement-convergence}, along the chosen subsequence,
\[\mathcal M(\beta_n\BetaBall;N,R)\longrightarrow\mathcal M(\mathcal P;N,R)\]
in Hausdorff distance. Approximate any $\mu\in\mathcal M(\mathcal P;N,R)$ by measures $\mu_n\in\mathcal M(\beta_n\BetaBall;N,R)$. Then \Cref{lem:star-measurement-upper} gives $\mu\in\mathcal M(\mathcal P_\lambda^\star;N,R)$. The reverse inclusion of $(N,R)$-measurements follows from the inclusion of pyramids. Thus the $(N,R)$-measurements agree for every $N,R$, and \Cref{cor:star-measurements-determine} yields $\mathcal P=\mathcal P_\lambda^\star$.

The space of pyramids is compact by \Cref{def:common-pyramid-weak-convergence}. Since every weak subsequential limit is $\mathcal P_\lambda^\star$, the entire sequence converges weakly to this pyramid. This completes the proof.
\end{proof}

\subsection{\texorpdfstring{Structure of $\mathcal P_\lambda^\star$}{Structure of the shifted-exponential railway pyramid}}
\label{sec:star-pyramid-structure}

The shifted-exponential railway-pyramid description in \Cref{thm:phase-diagram-draft} compares different values of $\lambda$ through branchwise contractions. Separation detects that these contractions give strict inclusions, and the discrete spaces dominated by the generators rule out a principal representation when $\lambda>0$.

% v2: prop:star-pyramid-parameter
\begin{prop}[Dependence on $\lambda$]
\label{prop:star-pyramid-parameter}
If $0\le\lambda<\lambda'$, then
\[\mathcal P_\lambda^\star\subsetneq\mathcal P_{\lambda'}^\star.\]
Moreover, for $N\ge1$ and $\kappa_N\coloneqq1/(N+1)$,
\[\lambda\le\operatorname{Sep}(\mathcal P_\lambda^\star;\underbrace{\kappa_N,\ldots,\kappa_N}_{N+1})\le\lambda+2\log\frac{N+1}{N}.\]
\end{prop}
\begin{proof}
Write points as $(\lambda+u,i)$. The map
\[(\lambda'+u,i)\longmapsto(\lambda+u,i)\]
is measure-preserving. It preserves same-branch distances and decreases cross-branch distances from $\lambda'+u+v$ to $\lambda+u+v$. Thus every generator with parameter $\lambda'$ dominates the corresponding generator with parameter $\lambda$, and $\mathcal P_\lambda^\star\subset\mathcal P_{\lambda'}^\star$.

Each branch of $T_{N+1,\lambda}(\kappa_N,\ldots,\kappa_N)$ has mass $\kappa_N$, and distinct branches have distance $\lambda$. Since pyramid separation approaches the mass parameters from below, \Cref{def:common-observable-separation} gives the lower bound $\lambda$.

For the upper bound, fix $r>\log((N+1)/N)$ and choose $\eta>0$ such that $N(\kappa_N-3\eta)>e^{-r}$. In every finite-branch railway mm-space, the set
\[C_r\coloneqq \{(\lambda+u,i)\mid0\le u\le r\}\]
has diameter at most $\lambda+2r$, and its complement has mass $e^{-r}$. Suppose that $N+1$ sets of mass at least $\kappa_N-3\eta$ have mutual distance greater than $\lambda+2r$. At most one can meet $C_r$. The remaining $N$ sets are pairwise disjoint subsets of its complement, but their total mass is at least $N(\kappa_N-3\eta)>e^{-r}$, a contradiction. Taking inverse images under a map from a finite-branch railway mm-space to an mm-space it dominates preserves the masses and the lower bounds on mutual distances. Hence the same upper bound holds for every mm-space dominated by a finite-branch railway mm-space.

Let $X\in\mathcal P_\lambda^\star$, and choose mm-spaces $Y_j$, each dominated by a finite-branch railway mm-space, such that $\Box(Y_j,X)\to0$. Suppose that $\operatorname{Sep}(X;\kappa_N-2\eta,\ldots,\kappa_N-2\eta)>\lambda+2r$. Choose $d$ strictly between these two values and Borel sets $A_0,\ldots,A_N$ whose mutual distances exceed $d$. For all sufficiently large $j$, choose a coupling $\pi_j$ of $X$ and $Y_j$ and a closed relation $S_j$ such that
\[1-\pi_j(S_j)<\eta/2,\qquad\operatorname{dis}S_j<\min\{\eta/2,d-\lambda-2r\}.\]
The projection to $Y_j$ of $S_j\cap(A_i\times Y_j)$ has mass greater than $\kappa_N-5\eta/2$. This projection is analytic and therefore measurable for the completed Borel measure. By inner regularity, it contains a Borel subset $B_i$ of mass at least $\kappa_N-3\eta$. The distortion bound gives $d_{Y_j}(B_i,B_{i'})>\lambda+2r$ for $i\ne i'$, contradicting the preceding upper bound. Therefore,
\[\operatorname{Sep}(X;\kappa_N-2\eta,\ldots,\kappa_N-2\eta)\le\lambda+2r.\]
Letting $\eta\downarrow0$ in the definition of pyramid separation and then $r\downarrow\log((N+1)/N)$ proves the required upper bound.

If $\lambda<\lambda'$, choose $N$ such that $2\log((N+1)/N)<\lambda'-\lambda$. The separation estimates for the two pyramids are then different, so the inclusion is strict. This completes the proof.
\end{proof}

% v2: cor:star-pyramid-nonprincipal
\begin{cor}[No associated mm-space for positive $\lambda$]
\label{cor:star-pyramid-nonprincipal}
If $\lambda>0$, then there is no mm-space $X$ such that $\mathcal P_\lambda^\star=\Pyr(X)$.
\end{cor}
\begin{proof}
Let $E_{m,\lambda}$ be the uniform $m$-point mm-space in which all distinct points have distance $\lambda$. The map from the equally weighted $T_{m,\lambda}$ to its branch labels is measure-preserving. It sends each branch to one point and decreases cross-branch distances to $\lambda$. Thus $E_{m,\lambda}\in\mathcal P_\lambda^\star$.

Suppose that $\mathcal P_\lambda^\star=\Pyr(X)$. Define $x\approx y$ if $x,y\in X$ can be joined by a finite chain whose every step has length less than $\lambda$. Each equivalence class is open, and separability of $X$ implies that there are at most countably many classes. Since $X=\operatorname{supp}\mu_X$, every class has positive measure. Fix one class $C$ and put $q\coloneqq\mu_X(C)>0$.

Choose $m$ such that $m^{-1}<q$. Since $E_{m,\lambda}\in\Pyr(X)$, there is a measure-preserving $1$-Lipschitz map $f\colon X\to E_{m,\lambda}$. Every step in a defining chain is shorter than the distance between distinct points of $E_{m,\lambda}$, so $f$ is constant on $C$. Thus $f_*\mu_X$ has an atom of mass at least $q>m^{-1}$ and cannot be the uniform measure. This is a contradiction.
	This completes the proof.
\end{proof}

\subsection{Phase transition property}

The weak limit yields the phase transition property once its observable diameter is finite at every fixed mass and finite-branch railway mm-spaces supply positive separation for every admissible tuple of masses.

% v2: thm:small-beta-ptp
\begin{lem}[Phase transition property criterion from a weak pyramid limit]
\label{lem:common-ptp-criterion}
Put $Z_n\coloneqq s_nX_n$ and suppose that $\Pyr(Z_n)\to\mathcal Q$ weakly as $n\to\infty$. Assume that $\operatorname{ObsDiam}(\mathcal Q;-\kappa)<+\infty$ for every $\kappa>0$. Suppose also that, for every $N$ and all positive numbers $\kappa_i$ satisfying $\sum_i\kappa_i<1$, there are $Y\in\mathcal Q$ and $\varepsilon>0$ such that
\[\operatorname{Sep}(Y;\kappa_0+\varepsilon,\ldots,\kappa_N+\varepsilon)>0.\]
Then $X_n$ has the phase transition property with critical scale $s_n$.
\end{lem}
\begin{proof}
The limit formulas in \Cref{prop:common-pyramid-limit-formulas}, together with $\operatorname{ObsDiam}(\mathcal Q;-(\kappa/2))<+\infty$, make the quantities $\operatorname{ObsDiam}(Z_n;-\kappa)$ uniformly bounded for all sufficiently large $n$, for every $\kappa>0$. If $\alpha_n/s_n\to0$, then
\[\operatorname{ObsDiam}((\alpha_n/s_n)Z_n;-\kappa)=(\alpha_n/s_n)\operatorname{ObsDiam}(Z_n;-\kappa)\to0.\]
This is the L\'evy-family condition for $((\alpha_n/s_n)Z_n)$.

Fix admissible $\kappa_0,\ldots,\kappa_N$ and choose $Y$ and $\varepsilon$ as in the hypothesis. Since $Y\in\mathcal Q$, \Cref{prop:common-pyramid-limit-formulas} gives some $\delta>0$ such that
\[\liminf_{n\to\infty}\operatorname{Sep}(Z_n;\kappa_0,\ldots,\kappa_N)\ge\delta.
\]
Indeed, the $\varepsilon$-buffer gives $\operatorname{Sep}(\mathcal Q;\kappa_0+\varepsilon/2,\ldots,\kappa_N+\varepsilon/2)>0$. Apply the separation limit formula at these masses with a decrement smaller than $\varepsilon/2$, and then use monotonicity in the mass parameters to obtain the displayed bound.
If $\alpha_n/s_n\to+\infty$, homogeneity of the distance then shows that $((\alpha_n/s_n)Z_n)$ dissipates at every positive scale and therefore infinitely dissipates. Since $(\alpha_n/s_n)Z_n=\alpha_nX_n$, the two conditions in the definition of the phase transition property follow. This completes the proof.
\end{proof}

\begin{thm}[Phase transition property in the railway-pyramid regime]
\label{thm:star-ptp}
Under the assumptions of \Cref{thm:phase-diagram-draft}, the sequence $\BetaBall$ has the phase transition property with critical scale $\beta_n$.
\end{thm}
\begin{proof}
The weak convergence $\Pyr(\beta_n\BetaBall)\to\mathcal P_\lambda^\star$ as $n\to\infty$ is supplied by \Cref{thm:phase-diagram-draft}. For the limiting pyramid, \Cref{lem:star-pyramid-finite-obsdiam} gives finite observable diameter at every fixed mass. Every finite-branch railway mm-space also belongs to $\mathcal P_\lambda^\star$ by \Cref{thm:phase-diagram-draft}. It follows from \Cref{lem:star-rich-separation} that separation is positive for every admissible choice of masses. The hypotheses of \Cref{lem:common-ptp-criterion} are now satisfied, and the result follows by taking $s_n=\beta_n$ there. This completes the proof.
\end{proof}

\section{Completion of the critical-limit classification}
\label{sec:sharp-phase-diagram}

\subsection{Diameter-at-most-one regime}
\label{sec:leb-phase}

At scale $\mathcal L_{n,\beta_n}$, fluctuations of the radial coordinate are negligible, so the angular geometry at the radius displayed below determines the limit. The exact hyperbolic distance between points at that radius carries spherical separation to a limit of $1$ for every admissible tuple of masses. We use the representative radius
\[\bar{\mathcal R}_{n,\beta}\coloneqq \operatorname{arcosh}\left(1+\frac n\beta\right).\]

% v2: lem:radial-width
% v2: lem:log-gamma-clt
\begin{lem}
\label{lem:common-log-gamma-clt}
Let $W_k\sim\Gamma(k,1)$ and put $Y_k\coloneqq(W_k-k)/\sqrt k$. Then, as $k\to+\infty$,
\[Y_k\overset{\mathrm d}{\longrightarrow}\gamma^1,\qquad\frac{W_k}{k}\overset{\mathrm P}{\longrightarrow}1,\qquad\log\frac{W_k}{k}\overset{\mathrm P}{\longrightarrow}0.\]
Moreover, the characteristic functions of $Y_k$ converge to that of $\gamma^1$ uniformly on bounded intervals. If $a_n,b_n\to+\infty$, the random variables $U_n\sim\Gamma(b_n,1)$ and $V_n\sim\Gamma(a_n,1)$ are independent, and $\lambda_n,\mu_n\ge0$ satisfy $\lambda_n+\mu_n=1$, then
\[\sqrt{\lambda_n}\frac{U_n-b_n}{\sqrt{b_n}}-\sqrt{\mu_n}\frac{V_n-a_n}{\sqrt{a_n}}\overset{\mathrm d}{\longrightarrow}\gamma^1.\]
\end{lem}
\begin{proof}
For $|z|\le1/2$,
\[\left|-\log(1-z)-z-\frac{z^2}{2}\right|\le\frac83|z|^3.\]
The characteristic function of $W_k$ is $(1-it)^{-k}$. Substituting $z=it/\sqrt k$ gives, for $|t|\le\sqrt k/2$,
\[\left|\log\mathbb E[e^{itY_k}]+\frac{t^2}{2}\right|\le\frac{8|t|^3}{3\sqrt k}.\]
The estimate shows that the characteristic functions converge to $e^{-t^2/2}$ uniformly on bounded intervals, and L\'evy's continuity theorem gives the first convergence. Since $\mathbb E[W_k]=\operatorname{Var}(W_k)=k$,
\[\mathbb P(|W_k/k-1|>\varepsilon)\le\frac1{k\varepsilon^2},\]
and the other two convergences follow from continuity of the logarithm.

Let $\varphi_k$ be the characteristic function of $Y_k$. The characteristic function of the final random variable in the statement is $\varphi_{b_n}(\sqrt{\lambda_n}t)\varphi_{a_n}(-\sqrt{\mu_n}t)$. The uniform convergence above and $\lambda_n+\mu_n=1$ show that it converges to $e^{-\lambda_nt^2/2}e^{-\mu_nt^2/2}=e^{-t^2/2}$. A second application of the continuity theorem gives the final convergence. This completes the proof.
\end{proof}

\begin{lem}
\label{lem:leb-radial-width}
If, as $n\to\infty$,
\[\mathcal A_{n,\beta_n}\to+\infty,\qquad\beta_n\mathcal L_{n,\beta_n}\to+\infty,\]
then
\[\frac{|\mathbf R_{n,\beta_n}-\bar{\mathcal R}_{n,\beta_n}|}{\mathcal L_{n,\beta_n}}\overset{\mathrm P}{\longrightarrow}0.\]
\end{lem}
\begin{proof}
\Cref{lem:common-scale-algebra} gives $\beta_n/\sqrt n\to0$ and $\mathcal L_{n,\beta_n}\to+\infty$. In particular, $\beta_n/n\to0$. Using \Cref{lem:common-radius-log} at $\mathbf V_{n,\beta_n}/(\mathbf U_n+\mathbf V_{n,\beta_n})$ and $2\beta_n/(n+2\beta_n)$ gives
\[\mathbf R_{n,\beta_n}-\bar{\mathcal R}_{n,\beta_n}=\log\frac{\mathbf U_n+\mathbf V_{n,\beta_n}}{n/2+\beta_n}-\log\frac{\mathbf V_{n,\beta_n}}{\beta_n}+\mathfrak r\left(\frac{\mathbf V_{n,\beta_n}}{\mathbf U_n+\mathbf V_{n,\beta_n}}\right)-\mathfrak r\left(\frac{2\beta_n}{n+2\beta_n}\right).\]
The Gamma law of large numbers in \Cref{lem:common-log-gamma-clt} gives $2\mathbf U_n/n\to1$ in probability. The nonnegative random variable $2\mathbf V_{n,\beta_n}/n$ has expectation $2\beta_n/n\to0$, so Markov's inequality shows that it converges to $0$ in probability. It follows that the first logarithmic term converges to $0$ in probability. The absolute value of the difference between the last two terms is less than $\log4$. It remains to prove that
\[\frac{\log(\mathbf V_{n,\beta_n}/\beta_n)}{\mathcal L_{n,\beta_n}}\overset{\mathrm P}{\longrightarrow}0.\]

From any subsequence, pass further so that $\beta_n$ tends to $0$, to some $\beta_*\in(0,+\infty)$, or to $+\infty$. If $\beta_n\to+\infty$, then \Cref{lem:common-log-gamma-clt} gives $\log(\mathbf V_{n,\beta_n}/\beta_n)\to0$ in probability. If $\beta_n\to\beta_*$, the characteristic functions of the gamma distributions show that $\mathbf V_{n,\beta_n}\Rightarrow\Gamma(\beta_*,1)$, and hence $\log(\mathbf V_{n,\beta_n}/\beta_n)$ is tight. Since $\mathcal L_{n,\beta_n}\to+\infty$, both cases give the required convergence.

If $\beta_n\to0$, then \Cref{lem:common-small-gamma} shows that $-\beta_n\log\mathbf V_{n,\beta_n}$ is tight. Since $\beta_n|\log\beta_n|\to0$,
\[\frac{|\log(\mathbf V_{n,\beta_n}/\beta_n)|}{\mathcal L_{n,\beta_n}}\le\frac{|\beta_n\log\mathbf V_{n,\beta_n}|+\beta_n|\log\beta_n|}{\beta_n\mathcal L_{n,\beta_n}}\overset{\mathrm P}{\longrightarrow}0.\]
The three cases and the subsequence criterion prove the assertion. This completes the proof.
\end{proof}

The formula in \Cref{lem:common-shell-transform} represents the shell metric as a metric transform of the Euclidean chord metric.

% v2: sec:preliminaries（metric transforms）
\begin{defn}[Metric transforms {\cite[Definition~1.1 and Proposition~4.1]{kazukawa2022metric}}]
\label{def:common-metric-transform}
A continuous function $F\colon[0,+\infty)\to[0,+\infty)$ is \emph{metric preserving} if $F\circ d$ is a metric whenever $d$ is a metric. For an mm-space $X$ and a pyramid $\mathcal P$, define their \emph{metric transforms} by
\[F(X)\coloneqq (X,F\circ d_X,\mu_X),\qquad F(\mathcal P)\coloneqq \overline{\bigcup_{X\in\mathcal P}\Pyr(F(X))}^{\,\Box}.\]
\end{defn}

% v2: lem:transform-sep
\begin{lem}
\label{lem:leb-transform-separation}
Let $F\colon[0,+\infty)\to[0,+\infty)$ be continuous, nondecreasing, and metric preserving. If an mm-space $X$ and positive numbers $\kappa_0,\ldots,\kappa_N$ satisfy $\operatorname{Sep}(X;\kappa_0,\ldots,\kappa_N)<+\infty$, then
\[\operatorname{Sep}(F(X);\kappa_0,\ldots,\kappa_N)=F\bigl(\operatorname{Sep}(X;\kappa_0,\ldots,\kappa_N)\bigr).\]
\end{lem}
\begin{proof}
Put $S\coloneqq\operatorname{Sep}(X;\kappa_0,\ldots,\kappa_N)$. For any admissible Borel sets $A_0,\ldots,A_N$, continuity and monotonicity of $F$ give
\[d_{F(X)}(A_i,A_j)=\inf_{x\in A_i,\,y\in A_j}F(d_X(x,y))=F(d_X(A_i,A_j)).\]
This bounds the minimum transformed distance for each admissible family by $F(S)$. Taking the supremum gives the upper bound.

If $S=0$, the reverse inequality follows from $F(0)=0$. Suppose that $S>0$ and fix $0<\chi<S$. By definition, an admissible family exists whose minimum distance is greater than $S-\chi$. Its transformed minimum distance is at least $F(S-\chi)$. Letting $\chi\downarrow0$ and using continuity of $F$ proves the reverse inequality. This completes the proof.
\end{proof}

% v2: prop:diameter-one-criterion
% v2: prop:diameter-one-criterion closes this via \cite[Proposition~8.5(1)]{shioya2016mmg}
\begin{lem}[$\delta$-dissipation and pyramid limits {\cite[Proposition~8.5(1)]{shioya2016mmg}}]
\label{lem:common-dissipation-limit}
For $\delta>0$ and a sequence of mm-spaces $(X_n)$, the following are equivalent.
\begin{enumerate}[label=\textup{(\roman*)}]
  \item $(X_n)$ $\delta$-dissipates.
  \item Every weak subsequential limit of $\Pyr(X_n)$ contains all mm-spaces of diameter at most $\delta$.
\end{enumerate}
\end{lem}

\begin{prop}[A separation criterion for $\Ple$]
\label{prop:leb-diameter-one-criterion}
Suppose that a sequence of mm-spaces $(X_n)$ satisfies, for every $N\ge1$ and every choice of positive numbers $\kappa_0,\ldots,\kappa_N$ with $\sum_{i=0}^N\kappa_i<1$,
\[\operatorname{Sep}(X_n;\kappa_0,\ldots,\kappa_N)\to1\qquad\text{as }n\to\infty.\]
Then
\[d_\Pi(\Pyr(X_n),\Ple)\to0.\]
\end{prop}
\begin{proof}
Let $\mathcal P$ be an arbitrary weak subsequential limit of $\Pyr(X_n)$, and relabel the corresponding subsequence by $n$. For every $0<\delta<1$, the hypothesis implies that $(X_n)$ $\delta$-dissipates, so \Cref{lem:common-dissipation-limit} shows that $\mathcal P$ contains every mm-space of diameter at most $\delta$. Let $Y$ be an mm-space of diameter at most $1$, and fix $c\in(0,1)$. Choosing $\delta<1$ larger than the diameter of $cY$ gives $cY\in\mathcal P$. For any $\eta>0$, choose a Borel subset of $Y$ with mass greater than $1-\eta$ and finite diameter $D$. Under the coupling induced by the identity map, the discrepancy between the two metrics on this subset is at most $(1-c)D$. Therefore,
\[\Box(cY,Y)\le\max\{\eta,(1-c)D\}.\]
Letting $c\uparrow1$ and then $\eta\downarrow0$ gives $\Box(cY,Y)\to0$. Since $\mathcal P$ is $\Box$-closed,
\[\Ple\subset\mathcal P.\]

Now fix positive numbers $\kappa_0,\ldots,\kappa_N$ whose sum is less than $1$. For every sufficiently small $\varepsilon>0$, one also has $\sum_i(\kappa_i-\varepsilon)<1$, and the separation of the pyramid associated with an mm-space equals the separation of the mm-space itself. The hypothesis and \Cref{prop:common-pyramid-limit-formulas} give
\[\operatorname{Sep}(\mathcal P;\kappa_0,\ldots,\kappa_N)=1.\]
If some $Y\in\mathcal P$ had diameter greater than $1$, choose $x,y\in Y$ with $d_Y(x,y)>1$ and $0<r<(d_Y(x,y)-1)/2$. Since $Y=\operatorname{supp}\mu_Y$, the two open $r$-balls have positive measure and mutual distance greater than $1$. Choose positive numbers $\kappa_0,\kappa_1$ smaller than the respective measures of the balls and with sum less than $1$. Then
\[\operatorname{Sep}(\mathcal P;\kappa_0,\kappa_1)\ge\operatorname{Sep}(Y;\kappa_0,\kappa_1)>1,\]
contradicting the preceding equality. Hence $\mathcal P\subset\Ple$.

Every weak subsequential limit is therefore $\Ple$. Compactness from \Cref{def:common-pyramid-weak-convergence} shows that the entire sequence converges to $\Ple$ in $d_\Pi$. This completes the proof.
\end{proof}

% v2: lem:deterministic-shell-limit-lebesgue
% v2: lem:exact-shell-transform
\begin{lem}
\label{lem:common-shell-transform}
For $R>0$, put
\[d_{n,R}(\theta,\eta)\coloneqq \dH((R,\theta),(R,\eta)),\qquad\Sigma_{n,R}\coloneqq (S^{n-1},d_{n,R},\sigma_{n-1}).\]
Then
\[d_{n,R}(\theta,\eta)=2\operatorname{arsinh}\left(\frac{\sinh R\,|\theta-\eta|}{2}\right),\qquad\Sigma_{n,R}=F_{\sinh R/\sqrt n}(\mathsf S_n).\]
\end{lem}
\begin{proof}
Putting $r=s=R$ in \eqref{eq:common-polar-distance} gives $\cosh d_{n,R}=1+\sinh^2R|\theta-\eta|^2/2$. The first identity follows from $\cosh(2\operatorname{arsinh}u)=1+2u^2$ and injectivity of $\cosh$ on $[0,+\infty)$. Transforming the metric of $\mathsf S_n$ by $F_{\sinh R/\sqrt n}$ gives the second identity. This completes the proof.
\end{proof}

\begin{lem}[Deterministic-shell convergence to $\Ple$]
\label{lem:leb-deterministic-shell}
Let $\rho_n>0$ be deterministic and put $A_{n,\rho_n}\coloneqq\frac{\sinh\rho_n}{\sqrt n}$. If $A_{n,\rho_n}\to+\infty$ as $n\to\infty$, then $A_{n,\rho_n}>1$ for all sufficiently large $n$, and
\[d_\Pi\left(\Pyr\bigl((2\log A_{n,\rho_n})^{-1}\Sigma_{n,\rho_n}\bigr),\Ple\right)\to0\qquad\text{as }n\to\infty.\]
\end{lem}
\begin{proof}
The inequality $A_{n,\rho_n}>1$ holds eventually because $A_{n,\rho_n}\to+\infty$. Fix positive numbers $\kappa_0,\ldots,\kappa_N$ with $\sum_i\kappa_i<1$. Combining \Cref{lem:common-spherical-block} with homogeneity under multiplication of the metric by $\sqrt n$ yields constants $0<c<C<+\infty$ such that, for all sufficiently large $n$,
\[c\le\operatorname{Sep}(\mathsf S_n;\kappa_0,\ldots,\kappa_N)\le C.\]
The formulas for the hyperbolic shell metric transform and separation in \Cref{lem:common-shell-transform,lem:leb-transform-separation} give
\[\operatorname{Sep}\left((2\log A_{n,\rho_n})^{-1}\Sigma_{n,\rho_n};\kappa_0,\ldots,\kappa_N\right)=\frac{F_{A_{n,\rho_n}}(\operatorname{Sep}(\mathsf S_n;\kappa_0,\ldots,\kappa_N))}{2\log A_{n,\rho_n}}.\]

For $a\ge1$ and $r\in[c,C]$, the identities $\operatorname{arsinh}u=\log(u+\sqrt{u^2+1})$ and $\sqrt{u^2+1}\le u+1$ give
\[2\log a+2\log(c/2)\le F_a(r)\le2\log a+2\log(C+1).\]
After division by $2\log A_{n,\rho_n}$, the middle ratio converges to $1$ uniformly for $r\in[c,C]$. Thus the displayed separation tends to $1$ for every admissible tuple of masses, and \Cref{prop:leb-diameter-one-criterion} gives the conclusion. This completes the proof.
\end{proof}

% v2: lem:lebesgue-radial-shell-comparison
% v2: lem:radial-window-shell-comparison
\begin{lem}[Radial-window shell comparison]
\label{lem:common-radial-window}
Let $h_n>0$ satisfy $\bar{\mathcal R}_{n,\beta_n}-h_n>0$ for all sufficiently large $n$, and put
\[p_n\coloneqq \mathbb P\bigl(|\mathbf R_{n,\beta_n}-\bar{\mathcal R}_{n,\beta_n}|>h_n\bigr).\]
Then, for every $s_n>0$ and all sufficiently large $n$,
\[\Box(s_n\BetaBall,s_n\Sigma_{n,\bar{\mathcal R}_{n,\beta_n}})\le\max\{2s_nh_n,p_n\}.\]
\end{lem}
\begin{proof}
Use the realization in \Cref{lem:common-radial-law} to couple $\mathbf X_{n,\beta_n}$ and $\Theta_n$ on the same probability space. The part of the coupling on which $|\mathbf R_{n,\beta_n}-\bar{\mathcal R}_{n,\beta_n}|\le h_n$ has mass $1-p_n$. For two pairs $(r,\theta)$ and $(s,\eta)$ in this part, the triangle inequality gives
\[\left|\dH((r,\theta),(s,\eta))-d_{n,\bar{\mathcal R}_{n,\beta_n}}(\theta,\eta)\right|\le|r-\bar{\mathcal R}_{n,\beta_n}|+|s-\bar{\mathcal R}_{n,\beta_n}|\le2h_n.\]
The hypothesis makes this relation a closed set avoiding the origin. Multiplying both metrics by $s_n$ and applying \Cref{def:common-box-concentration} proves the estimate. This completes the proof.
\end{proof}

\begin{lem}
\label{lem:leb-radial-shell-comparison}
If $\mathcal A_{n,\beta_n}\to+\infty$ and $\beta_n\mathcal L_{n,\beta_n}\to+\infty$ as $n\to\infty$, then
\[\Box\bigl(\mathcal L_{n,\beta_n}^{-1}\BetaBall,\mathcal L_{n,\beta_n}^{-1}\Sigma_{n,\bar{\mathcal R}_{n,\beta_n}}\bigr)\to0\qquad\text{as }n\to\infty.\]
\end{lem}
\begin{proof}
The radial-width estimate in \Cref{lem:leb-radial-width}, followed by a diagonal argument, provides a deterministic sequence $\varepsilon_n\downarrow0$ such that
\[\mathbb P\bigl(|\mathbf R_{n,\beta_n}-\bar{\mathcal R}_{n,\beta_n}|>\varepsilon_n\mathcal L_{n,\beta_n}\bigr)\to0.\]
The eventual positivity of $\mathcal L_{n,\beta_n}$ follows from \Cref{lem:common-scale-algebra}. The inequalities $\sinh R<e^R$ and $\sqrt n\ge1$, used with $R=\bar{\mathcal R}_{n,\beta_n}$, give
\[\mathcal L_{n,\beta_n}=2\log\frac{\sinh\bar{\mathcal R}_{n,\beta_n}}{\sqrt n}\le2\bar{\mathcal R}_{n,\beta_n}.\]
For all sufficiently large $n$, $\varepsilon_n\mathcal L_{n,\beta_n}<\bar{\mathcal R}_{n,\beta_n}$, so the radial window avoids the origin. Using \Cref{lem:common-radial-window} with $s_n=\mathcal L_{n,\beta_n}^{-1}$ and $h_n=\varepsilon_n\mathcal L_{n,\beta_n}$ yields
\[\Box\bigl(\mathcal L_{n,\beta_n}^{-1}\BetaBall,\mathcal L_{n,\beta_n}^{-1}\Sigma_{n,\bar{\mathcal R}_{n,\beta_n}}\bigr)\le\max\left\{2\varepsilon_n,\mathbb P\bigl(|\mathbf R_{n,\beta_n}-\bar{\mathcal R}_{n,\beta_n}|>\varepsilon_n\mathcal L_{n,\beta_n}\bigr)\right\}.\]
The right-hand side tends to $0$. This completes the proof.
\end{proof}

% v2: thm:lebesgue-type
% v2: lem:pyramid-rescaling-continuity
\begin{lem}
\label{lem:common-pyramid-rescaling-continuity}
For $c>0$ and a pyramid $\mathcal P$, put $c\mathcal P\coloneqq\{cY\mid Y\in\mathcal P\}$. If $c_n\to c>0$ and $d_\Pi(\mathcal P_n,\mathcal P)\to0$ as $n\to\infty$, then
\[d_\Pi(c_n\mathcal P_n,c\mathcal P)\to0.\]
In particular, if $d_\Pi(\Pyr(\Upsilon_n),\mathcal P)\to0$, then
\[d_\Pi(\Pyr(c_n\Upsilon_n),c\mathcal P)\to0.\]
\end{lem}
\begin{proof}
First suppose that $X_n\to X$ in $\Box$ and $c_n\to c>0$. Then $c_nX_n\to cX$ in $\Box$. Indeed, choose $K\subset X$ of finite diameter $D$ after discarding an arbitrarily small amount of mass. Let $S_n\subset X_n\times X$ be box relations whose mass and distortion errors tend to zero, and restrict $S_n$ to the pairs whose $X$-coordinate lies in $K$. The restricted relation loses at most the original relation loss plus $\mu_X(X\setminus K)$ in mass. For $(x_n,x),(y_n,y)$ in this relation,
\[|c_nd_{X_n}(x_n,y_n)-c\,d_X(x,y)|\le c_n\operatorname{dis}S_n+|c_n-c|D.\]
Letting $n\to\infty$ and then letting the discarded mass tend to zero proves the claim. The same argument applies to $c_n^{-1}\to c^{-1}$.

If $Y\in\mathcal P$ is approximated by $Y_n\in\mathcal P_n$, the preceding fact gives $c_nY_n\to cY$. Conversely, if $c_{n_j}Y_j\in c_{n_j}\mathcal P_{n_j}$ converges to $Z$, then $Y_j\to c^{-1}Z$. The upper-limit condition for weak convergence gives $c^{-1}Z\in\mathcal P$, and therefore $Z\in c\mathcal P$. Hence $c_n\mathcal P_n\to c\mathcal P$. The final assertion follows from $c_n\Pyr(\Upsilon_n)=\Pyr(c_n\Upsilon_n)$. This completes the proof.
\end{proof}

\begin{thm}[Convergence in the diameter-at-most-one regime]
\label{thm:leb-full}
If $\beta_n/\sqrt n\to0$ and $\beta_n\log n\to+\infty$ as $n\to\infty$, then
\[d_\Pi\bigl(\Pyr(\mathcal L_{n,\beta_n}^{-1}\BetaBall),\Ple\bigr)\to0\qquad\text{as }n\to\infty.\]
\end{thm}

\begin{proof}
By \Cref{lem:star-scale-identification}(ii), $\mathcal A_{n,\beta_n}\to+\infty$ and $\beta_n\mathcal L_{n,\beta_n}\to+\infty$ as $n\to\infty$.
At the representative radius,
\[\frac{\sinh\bar{\mathcal R}_{n,\beta_n}}{\sqrt n}=\mathcal A_{n,\beta_n},\qquad2\log\mathcal A_{n,\beta_n}=\mathcal L_{n,\beta_n}.\]
The deterministic-shell limit in \Cref{lem:leb-deterministic-shell} shows that $\Pyr(\mathcal L_{n,\beta_n}^{-1}\Sigma_{n,\bar{\mathcal R}_{n,\beta_n}})$ converges weakly to $\Ple$. The box distance between the two mm-spaces tends to $0$ by \Cref{lem:leb-radial-shell-comparison}, so \Cref{cor:common-box-limit} proves the general assertion. This completes the proof.
\end{proof}

% v2: thm:lebesgue-ptp
% v2: lem:admissible-critical-pyramids
\begin{lem}[Standard critical pyramids]
\label{lem:common-admissible-critical-pyramids}
The pyramids $\PGauss$, $F_a(\PGauss)$ for every $a>0$, and $\Ple$ satisfy the two hypotheses on the critical pyramid in \Cref{lem:common-ptp-criterion}.
\end{lem}
\begin{proof}
For every $1$-Lipschitz function $f$ on a finite-dimensional standard Gaussian space and every median $m_f$,
\[\gamma^k\{|f-m_f|\ge r\}\le2e^{-r^2/2}\]
implies, for $0<\kappa<1$,
\[\operatorname{ObsDiam}(\PGauss;-\kappa)\le2\sqrt{2\log(2/\kappa)}<+\infty\]
after passage to the closure using \Cref{prop:common-pyramid-limit-formulas}. For an admissible tuple $\kappa_0,\ldots,\kappa_N$, the condition $\sum_i\kappa_i<1$ allows us to choose compact intervals in the one-dimensional Gaussian space that have respective masses greater than $\kappa_i$ and are mutually separated by positive distances. Thus $\PGauss$ also satisfies the second condition.

Since $F_a(t)\le at$, one has $F_a(\PGauss)\subset a\PGauss$, and \Cref{lem:common-obsdiam-basic} gives the first condition. Applying $F_a$ to the distances between the intervals above preserves positive separation. Finally, every element of $\Ple$ has diameter at most $1$, which gives the first condition for $\Ple$. For the second, take an $(N+1)$-point mm-space whose distinct points have distance $1$ and whose point masses are greater than the corresponding $\kappa_i$. This completes the proof.
\end{proof}

\begin{thm}[Phase transition property in the diameter-at-most-one regime]
\label{thm:leb-ptp}
Under the general assumptions of \Cref{thm:leb-full}, the sequence $\BetaBall$ has the phase transition property with critical scale $\mathcal L_{n,\beta_n}^{-1}$.
\end{thm}
\begin{proof}
\Cref{thm:leb-full} identifies $\Ple$ as the weak limit of $\Pyr(\mathcal L_{n,\beta_n}^{-1}\BetaBall)$ as $n\to\infty$. For this pyramid, \Cref{lem:common-admissible-critical-pyramids} supplies both finite observable diameter and positive separation. Setting $s_n=\mathcal L_{n,\beta_n}^{-1}$ in \Cref{lem:common-ptp-criterion} yields the result. This completes the proof.
\end{proof}

\subsection{Nonlinear Gaussian-transform regime}
\label{sec:cross-phase}

A finite positive limit of the amplification parameter governs this regime. Continuity of the normalized transforms identifies the limit of the representative shells, whereas the radial central limit theorem shows that the radial width vanishes in the original metric.

% v2: lem:crossover-shell-limit
% v2: lem:normalized-shell-transform-continuity; closes via \cite[Corollary~1.4(A) and Proposition~4.1]{kazukawa2022metric}
\begin{lem}[Continuity of normalized hyperbolic shell metric transforms]
\label{lem:common-shell-transform-continuity}
For $a>0$, put $H_a\coloneqq a^{-1}F_a$, and set $H_0(t)\coloneqq t$. Each $H_a$, $a\ge0$, is continuous, nondecreasing, and metric preserving. Suppose that $u_n,c_n>0$ satisfy
\[u_n\to A\in[0,+\infty),\qquad c_n\to c\in(0,+\infty)\]
as $n\to\infty$. Then
\[\Pyr(c_nH_{u_n}(\mathsf S_n))\longrightarrow cH_A(\PGauss)\]
weakly as $n\to\infty$.
\end{lem}
\begin{proof}
For $a>0$,
\[H_a'(t)=\frac1{\sqrt{1+(at/2)^2}},\qquad H_a''(t)=-\frac{a^2t/4}{(1+(at/2)^2)^{3/2}}\le0.\]
The derivative formulas show that $H_a$ is nondecreasing and concave, with $H_a(0)=0$. Since the increments of a concave function decrease with the base point,
\[H_a(s+t)-H_a(s)\le H_a(t)-H_a(0)=H_a(t).\]
This proves that $H_a$ is subadditive and preserves the triangle inequality. The assertion for $H_0$ is immediate.

For every $t>0$, the expression
\[H_a(t)=t\frac{\operatorname{arsinh}(at/2)}{at/2}\]
extends continuously to $a=0$, and therefore $c_nH_{u_n}(t)\to cH_A(t)$. Each transform is also nondecreasing and has supremum $+\infty$. The metric-transform continuity theorem \cite[Corollary~1.4(A)]{kazukawa2022metric} states that pointwise convergence of continuous, nondecreasing, metric-preserving functions, together with the corresponding condition on their suprema, preserves weak convergence after metric transformation. It follows from \Cref{set:common-hyperbolic-polar} that the transformed pyramids converge weakly to $cH_A(\PGauss)$. Finally, the identity $F(\Pyr(X))=\Pyr(F(X))$ for transforms $F$ satisfying the same hypotheses is given by \cite[Proposition~4.1]{kazukawa2022metric}. Thus the transformed pyramid on the left is $\Pyr(c_nH_{u_n}(\mathsf S_n))$. This completes the proof.
\end{proof}

\begin{lem}[Finite-parameter shell limit]
\label{lem:cross-shell-limit}
Let $\rho_n>0$ be deterministic. If, for some $A\in(0,+\infty)$,
\[\frac{\sinh\rho_n}{\sqrt n}\to A\]
as $n\to\infty$, then
\[d_\Pi\bigl(\Pyr(\Sigma_{n,\rho_n}),F_A(\PGauss)\bigr)\to0.\]
\end{lem}
\begin{proof}
Since $F_c=cH_c$ for $c>0$, apply \Cref{lem:common-shell-transform-continuity} with both parameter sequences equal to $\sinh\rho_n/\sqrt n$ and with limiting coefficient $A$. This gives
\[\Pyr\left(F_{\sinh\rho_n/\sqrt n}(\mathsf S_n)\right)\longrightarrow F_A(\PGauss)\]
weakly. By \Cref{lem:common-shell-transform},
\[\Sigma_{n,\rho_n}=F_{\sinh\rho_n/\sqrt n}(\mathsf S_n).\]
The conclusion follows because $d_\Pi$ metrizes weak convergence.
	This completes the proof.
\end{proof}

% v2: lem:crossover-scale-algebra
\begin{lem}
\label{lem:cross-scale-algebra}
If $\mathcal A_{n,\beta_n}\to a\in(0,+\infty)$ as $n\to\infty$, then
\[\beta_n\to+\infty,\qquad\bar{\mathcal R}_{n,\beta_n}\to+\infty.\]
\end{lem}
\begin{proof}
In the identity
\[\mathcal A_{n,\beta_n}^2=\frac n{\beta_n^2}+\frac2{\beta_n}\to a^2\]
both terms are nonnegative. If $\beta_n$ did not tend to $+\infty$, there would be an $M>0$ and a subsequence along which $\beta_n\le M$. On that subsequence, $n/\beta_n^2\ge n/M^2\to+\infty$, a contradiction. Hence $\beta_n\to+\infty$. The identity above and $2/\beta_n\to0$ now give
\[\frac{\sqrt n}{\beta_n}\to a,\qquad\frac n{\beta_n}=\sqrt n\frac{\sqrt n}{\beta_n}\to+\infty.\]
Finally, $\cosh\bar{\mathcal R}_{n,\beta_n}=1+n/\beta_n\to+\infty$. Since $\cosh$ is increasing on $[0,+\infty)$, we obtain $\bar{\mathcal R}_{n,\beta_n}\to+\infty$.
	This completes the proof.
\end{proof}

% v2: lem:crossover-radial-shell-comparison
% v2: prop:radial-clt
\begin{prop}[Radial central limit theorem]
\label{prop:common-radial-clt}
If $\beta_n\to+\infty$ as $n\to\infty$, then
\[\sqrt{\beta_n}(\mathbf R_{n,\beta_n}-\bar{\mathcal R}_{n,\beta_n})\overset{\mathrm d}{\longrightarrow}\gamma^1.\]
\end{prop}
\begin{proof}[Proof sketch]
Divide the subsequential limits of $n/(2\beta_n)$ into the cases $0$, a finite positive value, and $+\infty$. In these three cases, approximate the half-angle formula by a square-root expression, a two-variable Taylor expansion, and a logarithm, respectively. Applying \Cref{lem:common-log-gamma-clt} to the first-order fluctuations yields the same standard Gaussian limit in every case. The full argument is given in \Cref{sec:app-radial-clt}. This completes the proof.
\end{proof}

\begin{lem}
\label{lem:cross-radial-shell-comparison}
If $\mathcal A_{n,\beta_n}\to a\in(0,+\infty)$ as $n\to\infty$, then
\[\Box(\BetaBall,\Sigma_{n,\bar{\mathcal R}_{n,\beta_n}})\to0.\]
\end{lem}
\begin{proof}
\Cref{lem:cross-scale-algebra} gives $\beta_n\to+\infty$ and $\bar{\mathcal R}_{n,\beta_n}\to+\infty$. The radial central limit theorem in \Cref{prop:common-radial-clt} gives
\[\sqrt{\beta_n}(\mathbf R_{n,\beta_n}-\bar{\mathcal R}_{n,\beta_n})\Rightarrow\gamma^1.\]
The sequence on the left is tight, and $\beta_n^{-1/2}\to0$. Hence $\mathbf R_{n,\beta_n}-\bar{\mathcal R}_{n,\beta_n}\to0$ in probability. A diagonal choice gives $h_n\downarrow0$ such that
\[\mathbb P(|\mathbf R_{n,\beta_n}-\bar{\mathcal R}_{n,\beta_n}|>h_n)\to0\]
as $n\to\infty$. Since the representative radius tends to $+\infty$, we have $h_n<\bar{\mathcal R}_{n,\beta_n}$ for all sufficiently large $n$. The radial-window estimate in \Cref{lem:common-radial-window}, used with $s_n=1$, yields
\[\Box(\BetaBall,\Sigma_{n,\bar{\mathcal R}_{n,\beta_n}})\le\max\left\{2h_n,\mathbb P(|\mathbf R_{n,\beta_n}-\bar{\mathcal R}_{n,\beta_n}|>h_n)\right\}\to0.\]
This completes the proof.
\end{proof}

% v2: thm:crossover
\begin{thm}[Limit in the nonlinear Gaussian-transform regime]
\label{thm:cross-full}
If $\sqrt n/\beta_n\to a\in(0,+\infty)$ as $n\to\infty$, then
\[d_\Pi\bigl(\Pyr(\BetaBall),F_a(\PGauss)\bigr)\to0.\]
\end{thm}

\begin{proof}
By \Cref{lem:star-scale-identification}(iii), $\mathcal A_{n,\beta_n}\to a$ as $n\to\infty$.
The transform parameter of the representative shell is
\[\frac{\sinh\bar{\mathcal R}_{n,\beta_n}}{\sqrt n}=\mathcal A_{n,\beta_n}\to a.\]
The shell-limit result in \Cref{lem:cross-shell-limit} shows that $\Pyr(\Sigma_{n,\bar{\mathcal R}_{n,\beta_n}})$ converges weakly to $F_a(\PGauss)$. On the other hand, \Cref{lem:cross-radial-shell-comparison} gives
\[\Box(\BetaBall,\Sigma_{n,\bar{\mathcal R}_{n,\beta_n}})\to0.\]
The box-limit statement in \Cref{cor:common-box-limit} then gives the same limit for $\Pyr(\BetaBall)$.
	This completes the proof.
\end{proof}

% v2: thm:crossover-ptp
\begin{thm}[Phase transition property in the nonlinear Gaussian-transform regime]
\label{thm:cross-ptp}
If $\sqrt n/\beta_n\to a\in(0,+\infty)$ as $n\to\infty$, then $\BetaBall$ has the phase transition property with critical scale $1$.
\end{thm}
\begin{proof}
Here $\Pyr(\BetaBall)$ converges weakly to $F_a(\PGauss)$ by \Cref{thm:cross-full}. The two required properties of this limit, finite observable diameter and positive separation, are given by \Cref{lem:common-admissible-critical-pyramids}. The assertion now follows from \Cref{lem:common-ptp-criterion} with $s_n=1$. This completes the proof.
\end{proof}

\subsection{Gaussian regime}
\label{sec:gauss-phase}

At the critical scale, a one-dimensional radial component remains after the angular metric is frozen at the representative shell. This component converges to a Gaussian line whose coefficient may vary with $n$. Together with the spherical factor, the line is absorbed into the Gaussian pyramid. Comparison with the original hyperbolic distance is made through a truncated metric on a radial window.

\subsubsection*{Warped-product component}

% v2: sec:gaussian-side (warped-product setup, L1033--1049)
Substitute $r=\bar{\mathcal R}_{n,\beta_n}+\mathcal A_{n,\beta_n}u$ and $s=\bar{\mathcal R}_{n,\beta_n}+\mathcal A_{n,\beta_n}v$ into the polar distance formula \eqref{eq:common-polar-distance}. This gives
\[
\begin{split}
\cosh\dH((r,\theta),(s,\eta))
&=\cosh\bigl(\mathcal A_{n,\beta_n}(u-v)\bigr)\\
&\quad+\frac12\sinh\bigl(\bar{\mathcal R}_{n,\beta_n}+\mathcal A_{n,\beta_n}u\bigr)\sinh\bigl(\bar{\mathcal R}_{n,\beta_n}+\mathcal A_{n,\beta_n}v\bigr)|\theta-\eta|^2.
\end{split}
\]
The first term yields the radial component $|u-v|$ after scaling the distance by $\mathcal A_{n,\beta_n}^{-1}$. At $u=v=0$, the identity $\sinh\bar{\mathcal R}_{n,\beta_n}=\mathcal A_{n,\beta_n}\sqrt n$ gives the normalized shell component
\[\mathcal A_{n,\beta_n}^{-1}d_{n,\bar{\mathcal R}_{n,\beta_n}}(\theta,\eta)=H_{\mathcal A_{n,\beta_n}}(\sqrt n|\theta-\eta|).\]
For bounded $u$ and $v$, each hyperbolic sine in the second term has ratio tending to one relative to its representative-shell value. The comparison metric is therefore the $\ell_2$-sum of the radial and normalized shell components; the uniform distortion estimate is proved in \Cref{lem:gauss-warped-product-comparison}.

The scaled space $\mathcal A_{n,\beta_n}^{-1}\BetaBall$ admits, after truncation, the following product approximation. Put
\[J_n\coloneqq \left[-\frac{\bar{\mathcal R}_{n,\beta_n}}{\mathcal A_{n,\beta_n}},+\infty\right),\qquad\nu_n\coloneqq \bigl(\mathcal A_{n,\beta_n}^{-1}(r-\bar{\mathcal R}_{n,\beta_n})\bigr)_*\mu_{n,\beta_n}\]
and define a metric $\delta_n$ on $J_n\times S^{n-1}$ by
\[\delta_n((u,\theta),(v,\eta))^2\coloneqq |u-v|^2+H_{\mathcal A_{n,\beta_n}}(\sqrt n|\theta-\eta|)^2\]
and the comparison mm-space by
\[Y_n\coloneqq (J_n\times S^{n-1},\delta_n,\nu_n\otimes\sigma_{n-1}).\]

% v2: lem:warped-product-gaussian-limit
% v2: lem:box-common-product
\begin{lem}
\label{lem:common-box-common-product}
For mm-spaces $X$ and $Q$, write $X\times_2Q$ for their product equipped with the $\ell_2$-product metric and product measure. Then, for any mm-spaces $X,Y,Q$,
\[\Box(X\times_2Q,Y\times_2Q)\le\Box(X,Y).\]
\end{lem}
\begin{proof}
Fix $\varepsilon>\Box(X,Y)$, and choose a coupling $\pi$ and a closed set $S\subset X\times Y$ such that $1-\pi(S)\le\varepsilon$ and $\operatorname{dis}S\le\varepsilon$. Couple the two copies of $Q$ diagonally, and put
\[\widetilde S\coloneqq \{((x,q),(y,q))\mid(x,y)\in S,\ q\in Q\}.\]
This closed relation has mass at least $1-\varepsilon$. For $r,s,t\ge0$, the inequality
\[|\sqrt{r^2+t^2}-\sqrt{s^2+t^2}|\le|r-s|\]
implies $\operatorname{dis}\widetilde S\le\operatorname{dis}S\le\varepsilon$. Letting $\varepsilon\downarrow\Box(X,Y)$ proves the assertion.
	This completes the proof.
\end{proof}

% v2: lem:one-coordinate-sphere-absorption
\begin{lem}[Gaussian limit after adjoining one coordinate]
\label{lem:common-one-coordinate-absorption}
Suppose that $v_n\to0$ as $n\to\infty$, and let $c_n\in[0,1]$ be arbitrary. Put $Q_n\coloneqq H_{v_n}(\mathsf S_n)$, and define the mm-space
\[\Gamma^1_{c_n^2}\coloneqq (\operatorname{supp}(m_{c_n})_*\gamma^1,|\cdot|,(m_{c_n})_*\gamma^1)\]
where $m_{c_n}(x)\coloneqq c_nx$. Thus $\Gamma^1_{c^2}$ is the centered Gaussian line of variance $c^2$, in agreement with \cite{kazukawa2022metric}, with $\Gamma^1_1=\Gamma^1$ and $\Gamma^1_0$ the one-point space. Then, for the product equipped with the $\ell_2$-product metric and product measure,
\[\Pyr(\Gamma^1_{c_n^2}\times_2Q_n)\longrightarrow\PGauss\]
weakly as $n\to\infty$.
\end{lem}
\begin{proof}
Since $H_a(t)\le t$ and $c_n\le1$,
\[Q_n\prec\mathsf S_n,\qquad \Gamma^1_{c_n^2}\prec\Gamma^1_1,\qquad \Gamma^1_{c_n^2}\times_2Q_n\prec\Gamma^1_1\times_2\mathsf S_n.\]
Let $G\sim\gamma^1$ and a standard Gaussian vector $G^{(n)}\in\R^n$ be independent, and put
\[P_n\coloneqq \left(G,\frac{\sqrt nG^{(n)}}{|G^{(n)}|}\right),\qquad \mathbf Z_n\coloneqq \frac{\sqrt n(G,G^{(n)})}{\sqrt{G^2+|G^{(n)}|^2}}.\]
The random vector $P_n$ realizes $\Gamma^1_1\times_2\mathsf S_n$, while $\mathbf Z_n$ realizes $S^n$ of radius $\sqrt n$. Since $|G^{(n)}|^2/n\to1$ in probability and $G$ is tight, the direct estimate
\[|P_n-\mathbf Z_n|\le |G|\left|1-\frac{\sqrt n}{\sqrt{G^2+|G^{(n)}|^2}}\right|+\frac{\sqrt nG^2}{\sqrt{G^2+|G^{(n)}|^2}(\sqrt{G^2+|G^{(n)}|^2}+|G^{(n)}|)}\]
has a right-hand side that converges to $0$ in probability. Apply \Cref{lem:common-ambient-box} and \Cref{set:common-hyperbolic-polar} in dimension $n+1$. The difference between the radii $\sqrt n$ and $\sqrt{n+1}$ is absorbed by \Cref{cor:common-box-limit}, giving
\[\Pyr(\Gamma^1_1\times_2\mathsf S_n)\longrightarrow\PGauss\]
weakly.

\Cref{lem:common-shell-transform-continuity} gives $\Pyr(Q_n)\to\PGauss$ weakly, while projection gives $Q_n\prec\Gamma^1_{c_n^2}\times_2Q_n$. It follows that, for every $N\ge1$ and $R>0$,
\[\mathcal M(Q_n;N,R)\subset\mathcal M(\Gamma^1_{c_n^2}\times_2Q_n;N,R)\subset\mathcal M(\Gamma^1_1\times_2\mathsf S_n;N,R).\]
The two outer sets converge in Hausdorff distance to the same set $\mathcal M(\PGauss;N,R)$, so the middle set does as well. The weak-convergence criterion in \Cref{def:common-measurement-convergence} proves the assertion. This completes the proof.
\end{proof}

\begin{lem}[Warped-product Gaussian limit]
\label{lem:gauss-warped-product-limit}
If $\mathcal A_{n,\beta_n}\to0$ as $n\to\infty$, then
\[d_\Pi(\Pyr(Y_n),\PGauss)\to0.\]
\end{lem}
\begin{proof}
Since $\mathcal A_{n,\beta_n}^2\ge2/\beta_n$, we have $\beta_n\to+\infty$. By \Cref{prop:common-radial-clt},
\[d_P\left(\bigl(\sqrt{\beta_n}(r-\bar{\mathcal R}_{n,\beta_n})\bigr)_*\mu_{n,\beta_n},\gamma^1\right)\to0.\]
On the other hand, the exact identity
\[\mathcal A_{n,\beta_n}^{-1}(r-\bar{\mathcal R}_{n,\beta_n})=\sqrt{\frac{\beta_n}{n+2\beta_n}}\,\sqrt{\beta_n}(r-\bar{\mathcal R}_{n,\beta_n})\]
holds, and the coefficient is at most $1/\sqrt2$. Since multiplication by this coefficient is $1$-Lipschitz,
\[d_P\left(\nu_n,\left(x\mapsto\sqrt{\frac{\beta_n}{n+2\beta_n}}x\right)_*\gamma^1\right)\to0.\]
The coupling and ambient-box estimates in \Cref{lem:common-prokhorov-coupling,lem:common-ambient-box} show that the box distance between the one-dimensional factor $(J_n,|\cdot|,\nu_n)$ and the Gaussian line on the right also tends to $0$.

The angular factor is $H_{\mathcal A_{n,\beta_n}}(\mathsf S_n)$. The one-coordinate absorption result in \Cref{lem:common-one-coordinate-absorption}, with the parameter sequence $\mathcal A_{n,\beta_n}\to0$ and the coefficient sequence $\sqrt{\beta_n/(n+2\beta_n)}\in[0,1]$, gives weak convergence of the pyramid of the $\ell_2$-product of the corresponding Gaussian line and the angular factor to $\PGauss$. The common-product estimate in \Cref{lem:common-box-common-product} bounds the box distance between this product and $Y_n$ by the box distance between their one-dimensional factors. The conclusion now follows from \Cref{cor:common-box-limit}. This completes the proof.
\end{proof}

% v2: lem:warped-product-comparison
\begin{lem}[Truncated warped-product comparison]
\label{lem:gauss-warped-product-comparison}
If $\mathcal A_{n,\beta_n}\to0$ as $n\to\infty$, then, for every $L>0$,
\[\Box\bigl((\mathcal A_{n,\beta_n}^{-1}\BetaBall)^{(L)},Y_n^{(L)}\bigr)\to0.\]
\end{lem}
\begin{proof}[Proof sketch]
Pull the scaled hyperbolic distance back to $J_n\times S^{n-1}$ by the radial-coordinate map and compare its truncation uniformly with $\delta_n\wedge L$ on bounded radial windows. Tightness of $(\nu_n)$ allows us to discard the mass outside such a window, and coupling the two spaces by their common product measure yields convergence in box distance. The full argument is given in \Cref{sec:app-warped-comparison}. This completes the proof.
\end{proof}

\subsubsection*{Assembly}

% v2: thm:gaussian-side-full
\begin{thm}[Limit in the Gaussian regime]
\label{thm:gauss-full}
If $\sqrt n/\beta_n\to0$ as $n\to\infty$, then
\[d_\Pi\bigl(\Pyr(\mathcal A_{n,\beta_n}^{-1}\BetaBall),\PGauss\bigr)\to0.\]
\end{thm}

\begin{proof}
By \Cref{lem:star-scale-identification}(iv), $\mathcal A_{n,\beta_n}\to0$ as $n\to\infty$.
Fix $N\ge1$ and $R>0$. Using the truncation invariance in \Cref{prop:common-truncation}, the measurement bound in \Cref{cor:common-box-limit}, and \Cref{lem:gauss-warped-product-comparison} with $L=2R$, we obtain
\[
\begin{split}
&\dHaus{d_P}\bigl(\mathcal M(\mathcal A_{n,\beta_n}^{-1}\BetaBall;N,R),\mathcal M(Y_n;N,R)\bigr)\\
&=\dHaus{d_P}\bigl(\mathcal M((\mathcal A_{n,\beta_n}^{-1}\BetaBall)^{(2R)};N,R),\mathcal M(Y_n^{(2R)};N,R)\bigr)\\
&\le\Box\bigl((\mathcal A_{n,\beta_n}^{-1}\BetaBall)^{(2R)},Y_n^{(2R)}\bigr)\to0.
\end{split}
\]
Together, \Cref{lem:gauss-warped-product-limit} and \Cref{def:common-measurement-convergence} show that $\mathcal M(Y_n;N,R)$ converges to $\mathcal M(\PGauss;N,R)$ with respect to $\dHaus{d_P}$. The triangle inequality gives the same limit for $\mathcal M(\mathcal A_{n,\beta_n}^{-1}\BetaBall;N,R)$. Since $N$ and $R$ were arbitrary, the weak-convergence criterion in \Cref{def:common-measurement-convergence} proves the assertion. This completes the proof.
\end{proof}

% v2: thm:gaussian-ptp
\begin{thm}[Phase transition property in the Gaussian regime]
\label{thm:gauss-ptp}
If $\sqrt n/\beta_n\to0$ as $n\to\infty$, then $\BetaBall$ has the phase transition property with critical scale $\mathcal A_{n,\beta_n}^{-1}$.
\end{thm}
\begin{proof}
Weak convergence of $\Pyr(\mathcal A_{n,\beta_n}^{-1}\BetaBall)$ to $\PGauss$ follows from \Cref{thm:gauss-full}. According to \Cref{lem:common-admissible-critical-pyramids}, the limiting pyramid has finite observable diameter and positive separation. The conclusion follows from \Cref{lem:common-ptp-criterion} upon choosing $s_n=\mathcal A_{n,\beta_n}^{-1}$. This completes the proof.
\end{proof}

% v2: sec:sharpness (completion of the critical-limit classification)
\begin{prop}[Subsequence exhaustion by the four critical-limit regimes]
\label{prop:sharp-exhaustiveness}
For every positive sequence $(\beta_n)$ and every subsequence, there is a further subsequence along which exactly one of the following four mutually exclusive cases holds as $n\to\infty$.
\begin{enumerate}[label=\textup{(\roman*)}]
  \item $\beta_n\to0$ and $\beta_n\log n\to\lambda\in[0,+\infty)$, and
  \[d_\Pi(\Pyr(\beta_n\BetaBall),\mathcal P_\lambda^\star)\to0.\]
  \item $\beta_n/\sqrt n\to0$ and $\beta_n\log n\to+\infty$, and
  \[d_\Pi(\Pyr(\mathcal L_{n,\beta_n}^{-1}\BetaBall),\Ple)\to0.\]
  \item For some $a\in(0,+\infty)$, $\sqrt n/\beta_n\to a$, and
  \[d_\Pi(\Pyr(\BetaBall),F_a(\PGauss))\to0.\]
  \item $\sqrt n/\beta_n\to0$, and
  \[d_\Pi(\Pyr(\mathcal A_{n,\beta_n}^{-1}\BetaBall),\PGauss)\to0.\]
\end{enumerate}
\end{prop}
\begin{proof}
Relabel the given subsequence by $n$. By compactness of the extended half-line $[0,+\infty]$, we may pass to a further subsequence along which $\mathcal A_{n,\beta_n}$ converges to $0$, to a finite positive value, or to $+\infty$.

If the limit is $0$, \Cref{lem:star-scale-identification}(iv) gives $\sqrt n/\beta_n\to0$, and \Cref{thm:gauss-full} supplies the pyramid convergence in case (iv). If it is $a\in(0,+\infty)$, \Cref{lem:star-scale-identification}(iii) gives $\sqrt n/\beta_n\to a$, and \Cref{thm:cross-full} supplies case (iii).

Suppose that $\mathcal A_{n,\beta_n}\to+\infty$. Then $\mathcal L_{n,\beta_n}=2\log\mathcal A_{n,\beta_n}$ is eventually positive. Passing to another subsequence, we may assume that the positive sequence $\beta_n\mathcal L_{n,\beta_n}$ has a limit in $[0,+\infty]$. If the limit is finite, \Cref{lem:star-scale-identification}(i) gives the parameter conditions in case (i), and \Cref{thm:phase-diagram-draft} gives its pyramid convergence. If the limit is $+\infty$, \Cref{lem:star-scale-identification}(ii) gives the parameter conditions in case (ii), and \Cref{thm:leb-full} gives its pyramid convergence. The original case split is mutually exclusive, and \Cref{lem:star-scale-identification} identifies it exactly with the four conditions stated above. This completes the proof.
\end{proof}

For a sequence satisfying the hypotheses of one of the four critical-limit regimes, both implications in \Cref{def:common-ptp} have converses at that regime's critical scale. At the endpoint rescaling ratios $0$ and $+\infty$, this follows from \Cref{thm:star-ptp,thm:leb-ptp,thm:cross-ptp,thm:gauss-ptp}. At a finite positive subsequential rescaling ratio, \Cref{lem:common-pyramid-rescaling-continuity,thm:phase-diagram-draft,thm:leb-full,thm:cross-full,thm:gauss-full} identify the rescaled critical pyramid. There, positive separation from \Cref{lem:star-rich-separation,lem:common-admissible-critical-pyramids}, transferred by \Cref{prop:common-pyramid-limit-formulas}, excludes the L\'evy property, and finite observable diameter from \Cref{lem:star-pyramid-finite-obsdiam,lem:common-admissible-critical-pyramids}, together with \Cref{lem:common-dissipation-limit}, excludes infinite dissipation.

\section{\texorpdfstring{Sharpness of the $\Ple$ boundary}{Sharpness of the diameter-at-most-one boundary}}
\label{sec:sharp-ple-boundary}

The classification in \Cref{prop:sharp-exhaustiveness} is subsequential: every subsequence has a further subsequence in one of the four critical-limit regimes. In contrast, \Cref{thm:sharp-width} gives a necessary-and-sufficient criterion for convergence of the original sequence to $\Ple$, and its necessity requires a separate argument.

For necessity at the $\Ple$ boundary, by which we mean the sharp asymptotic threshold $\beta_n\log n\to+\infty$ under $\beta_n/\sqrt n\to0$ for convergence to $\Ple$ rather than a topological boundary in parameter space, a subsequence on which $\beta_n\log n$ remains bounded returns to the railway-pyramid regime. The cases where its limit is zero or positive obstruct convergence to $\Ple$ by dissipation or by an unbounded one-branch generator, respectively.

% v2: thm:sharp-width
\begin{thm}[Sharpness of the criterion for convergence to $\Ple$]
\label{thm:sharp-width}
Assume that $\beta_n/\sqrt n\to0$ as $n\to\infty$. Then
\[d_\Pi\bigl(\Pyr(\mathcal L_{n,\beta_n}^{-1}\BetaBall),\Ple\bigr)\to0\]
if and only if
\[\beta_n\log n\to+\infty\]
as $n\to\infty$.
\end{thm}
\begin{proof}
Sufficiency is \Cref{thm:leb-full}. We prove necessity by contraposition. Suppose that $\beta_n\log n$ does not tend to $+\infty$. By \Cref{lem:common-scale-algebra}, $\mathcal A_{n,\beta_n}\to+\infty$ and $\mathcal L_{n,\beta_n}\to+\infty$. Hence \Cref{lem:star-scale-identification}(ii) shows that $\beta_n\mathcal L_{n,\beta_n}$ does not tend to $+\infty$, so there is a subsequence along which the product is bounded. Passing to a further subsequence, assume that
\[\beta_n\mathcal L_{n,\beta_n}\to\lambda\in[0,+\infty)\]
as $n\to\infty$. By \Cref{lem:star-scale-identification}(i), along this subsequence,
\[\beta_n\to0,\qquad\beta_n\log n\to\lambda.\]

If $\lambda=0$, then
\[\frac{\mathcal L_{n,\beta_n}^{-1}}{\beta_n}=\frac1{\beta_n\mathcal L_{n,\beta_n}}\to+\infty.\]
Applying \Cref{thm:star-ptp} at the critical scale $\beta_n$, we see that $\mathcal L_{n,\beta_n}^{-1}\BetaBall$ infinitely dissipates. By compactness of the space of pyramids, every subsequence of the associated pyramids has a weakly convergent further subsequence. Applying \Cref{lem:common-dissipation-limit} with $\delta=2$, we find that every such limit contains the two-point mm-space whose distinct points are at distance $2$. This mm-space has diameter greater than $1$ and therefore does not belong to $\Ple$; consequently, the resulting limit is not $\Ple$.

Now suppose that $\lambda>0$. By \Cref{thm:phase-diagram-draft},
\[\Pyr(\beta_n\BetaBall)\to\mathcal P_\lambda^\star.\]
Moreover,
\[\mathcal L_{n,\beta_n}^{-1}\BetaBall=\frac1{\beta_n\mathcal L_{n,\beta_n}}(\beta_n\BetaBall),\qquad\frac1{\beta_n\mathcal L_{n,\beta_n}}\to\frac1\lambda.\]
Hence \Cref{lem:common-pyramid-rescaling-continuity} gives
\[\Pyr(\mathcal L_{n,\beta_n}^{-1}\BetaBall)\to\lambda^{-1}\mathcal P_\lambda^\star.\]
By \Cref{thm:phase-diagram-draft,lem:star-pyramid,lem:star-map}, this limit contains the rescaled one-branch generator
\[([\lambda,+\infty),\lambda^{-1}|\cdot|,\nu_\lambda)\]
whose diameter is infinite because its support is unbounded. Thus this mm-space does not belong to $\Ple$, and neither does the limiting pyramid.

In both cases, the original sequence has a subsequence that does not converge to $\Ple$. Therefore, convergence to $\Ple$ requires $\beta_n\log n\to+\infty$. This completes the proof.
\end{proof}

\section{Appendix}
\label{sec:appendix}

\subsection{A hyperbolic-sphere application}
\label{sec:app-hyperbolic-sphere}

% v2: cor:hyperbolic-sphere-limit-lebesgue
\begin{cor}[A hyperbolic-sphere limit]
\label{cor:leb-hyperbolic-sphere}
For $n\ge2$, put
\[H_n\coloneqq \left(S^n,\frac{\operatorname{arsinh}(n|x-y|)}{\log\sqrt n},\sigma_n\right),\]
where $S^n\subset\R^{n+1}$ is the unit sphere, $|x-y|$ is the ambient Euclidean chord distance, and $\sigma_n$ is the normalized surface measure. Then
\[d_\Pi(\Pyr(H_n),\Ple)\to0\qquad\text{as }n\to\infty.\]
\end{cor}
\begin{proof}
The formula for the hyperbolic shell metric transform in \Cref{lem:common-shell-transform}, used in dimension $n+1$ at radius $\operatorname{arsinh}(2n)$, gives the shell metric
\[2\operatorname{arsinh}\left(\frac{2n|x-y|}{2}\right)=2\operatorname{arsinh}(n|x-y|)\]
and, since $(\log\sqrt n)^{-1}=2/\log n$,
\[H_n=(\log n)^{-1}\Sigma_{n+1,\operatorname{arsinh}(2n)}.\]
The parameter of this shell is $2n/\sqrt{n+1}\to+\infty$. \Cref{lem:leb-deterministic-shell} then gives
\[d_\Pi\left(\Pyr\left(\left(2\log\frac{2n}{\sqrt{n+1}}\right)^{-1}\Sigma_{n+1,\operatorname{arsinh}(2n)}\right),\Ple\right)\to0.\]
The ratio of the two normalizations satisfies
\[\frac{2\log(2n/\sqrt{n+1})}{\log n}=1+\frac{2\log2-\log(1+1/n)}{\log n}\to1.\]
Rescaling continuity in \Cref{lem:common-pyramid-rescaling-continuity}, with limiting coefficient $1$, gives $\Pyr(H_n)\to\Ple$.
	This completes the proof.
\end{proof}

\subsection{Profile compactness}
\label{sec:app-star-profile}

\begin{proof}[Proof of \Cref{lem:star-profile-compactness}]
Same-direction hyperbolic distances equal radial differences, so, for $s,t\in I$,
\[\|g_{k,\theta}(s)-g_{k,\theta}(t)\|_\infty\le|s-t|.\]
Thus $g_{k,\theta}\in K$. The family $K$ is uniformly bounded, equicontinuous, and closed, hence compact by the Arzel\`a--Ascoli theorem. Uniform continuity of $\Psi_k$ on the compact radial window makes $\theta\mapsto g_{k,\theta}$ continuous, so $q_k$ is a probability measure on $K$. After passing to a subsequence, let $q_k\Rightarrow q$.

Let $g,h\in\operatorname{supp}q$ be distinct, and choose $\tau>0$ small enough that their $\tau$-balls in $K$ are disjoint. Both balls have positive $q$-measure. By the Portmanteau theorem, for all sufficiently large $k$, the sets
\[U_k\coloneqq \{\theta\mid g_{k,\theta}\in B_K(g,\tau)\},\qquad V_k\coloneqq \{\theta\mid g_{k,\theta}\in B_K(h,\tau)\}\]
have measures at least positive constants $\kappa_g$ and $\kappa_h$, respectively. By \Cref{lem:common-spherical-two-set}, we may choose $\theta_k\in U_k$ and $\eta_k\in V_k$ such that
\[|\theta_k-\eta_k|\le\frac2{\sqrt{(n_k-1)\min\{\kappa_g,\kappa_h\}}}+\frac1{\sqrt{n_k}}\le\frac C{\sqrt{n_k}}\]
where $C\coloneqq2\sqrt{2/\min\{\kappa_g,\kappa_h\}}+1$.

By \Cref{lem:star-cross-branch-asymptotic}(ii), there is a sequence $\varepsilon_k\downarrow0$ such that, for every $s,t\in I$,
\[\|g(s)-h(t)\|_\infty\le2\tau+\|g_{k,\theta_k}(s)-g_{k,\eta_k}(t)\|_\infty\le2\tau+s+t-\lambda+\varepsilon_k.\]
Letting $k\to+\infty$ and then $\tau\downarrow0$ gives the cross-branch constraint.

Put $\zeta_k\coloneqq(\beta_{n_k}r)_*\mu_{n_k,\beta_{n_k}}$. By \Cref{prop:star-radial}, $\zeta_k\Rightarrow\nu_\lambda$. Since $\nu_\lambda$ assigns no mass to the endpoints of $I$, we have $\zeta_k|_I\Rightarrow\nu_\lambda|_I$. Radial--angular independence from \Cref{lem:common-radial-law} gives
\[(\Psi_k)_*\bigl(\mu_{n_k,\beta_{n_k}}|_{\{\beta_{n_k}r\in I\}}\bigr)=E_*\bigl((\zeta_k|_I)\otimes q_k\bigr).\]
The evaluation map $E\colon I\times K\to B_R^N$ is continuous, so the right-hand side converges weakly to $\mu^I$. The left-hand side is bounded above by $(\Psi_k)_*\mu_{n_k,\beta_{n_k}}$. Testing against arbitrary nonnegative continuous functions and passing to the limit gives $\mu^I\le\mu$. Since $\mu^I$ has total mass $\nu_\lambda(I)$, the remaining mass is $1-\nu_\lambda(I)$. This completes the proof.
\end{proof}

\subsection{\texorpdfstring{Upper inclusion for $(N,R)$-measurements}{Upper inclusion for (N,R)-measurements}}
\label{sec:app-star-measurement}

\begin{proof}[Proof of \Cref{lem:star-measurement-upper}]
By the support argument in \Cref{prop:common-truncation}, there is a map $\Psi_k\colon Z_{n_k}\to B_R^N$. It is $1$-Lipschitz and satisfies $\mu_k=(\Psi_k)_*\mu_{n_k,\beta_{n_k}}$. Fix $0<\varepsilon<1$, and choose $\delta>0$ and $M>\lambda+\delta$ such that
\[I\coloneqq [\lambda+\delta,M],\qquad\nu_\lambda(I)>1-\varepsilon\]
Apply \Cref{lem:star-profile-compactness} and pass to a subsequence to obtain a limiting profile measure $q$ and
\[\mu^I\coloneqq E_*\bigl((\nu_\lambda|_I)\otimes q\bigr)\le\mu\]

Cover $\operatorname{supp}q$ by open $\varepsilon$-balls centered at distinct points $g_1,\ldots,g_m$. Assign each point to the first ball containing it and discard cells of zero $q$-measure. This gives a Borel map
\[\psi\colon\operatorname{supp}q\to\{g_1,\ldots,g_m\},\qquad\|g-\psi(g)\|_\infty<\varepsilon\quad(q\text{-a.e. }g)\]
and positive weights $p_i\coloneqq q(\psi^{-1}(g_i))$. Define $a(s,i)\coloneqq g_i(s)$ on $I\times\{1,\ldots,m\}$. The $1$-Lipschitz property of $g_i$ on each branch and the cross-branch constraint in \Cref{lem:star-profile-compactness} show that $a$ is $1$-Lipschitz from the finite-branch railway metric to $\|\cdot\|_\infty$.

Write $a=(a_1,\ldots,a_N)$ and define, on the whole space $T_{m,\lambda}(p_1,\ldots,p_m)$,
\[\widetilde a_\ell(x)\coloneqq \inf_{y\in I\times\{1,\ldots,m\}}\{a_\ell(y)+d_\lambda(x,y)\}\qquad(1\le\ell\le N)\]
The triangle inequality shows that every $\widetilde a_\ell$ is $1$-Lipschitz and agrees with $a_\ell$ on the window. Clipping each coordinate to $[-R,R]$ gives a $1$-Lipschitz map
\[\Phi\colon T_{m,\lambda}(p_1,\ldots,p_m)\to B_R^N\]
that agrees with $a$ on the window. Denote its pushforward measure by $\nu_\varepsilon$.

On the window, send $(s,g)$ to $(g(s),\Phi(s,i))$ whenever $\psi(g)=g_i$. The first marginal is $\mu^I$, and the second is the part of $\nu_\varepsilon$ inside the window. The paired points are at distance $\|g(s)-g_i(s)\|_\infty<\varepsilon$. The remaining measures $\mu-\mu^I$ and the part of $\nu_\varepsilon$ outside the window both have mass $1-\nu_\lambda(I)<\varepsilon$ and may be coupled arbitrarily. This produces a coupling in $\mathcal T(\mu,\nu_\varepsilon)$ and gives $d_P(\mu,\nu_\varepsilon)\le\varepsilon$. Moreover, $\nu_\varepsilon\in\mathcal M(\mathcal P_\lambda^\star;N,R)$.

For $\varepsilon_j\downarrow0$, the measures $\nu_{\varepsilon_j}$ converge to $\mu$ in Prokhorov distance. By \Cref{lem:common-prokhorov-coupling,lem:common-ambient-box}, the mm-spaces determined by these measures converge in box distance to $(\operatorname{supp}\mu,\|\cdot\|_\infty,\mu)$. Since $\mathcal P_\lambda^\star$ is box closed, $\mu\in\mathcal M(\mathcal P_\lambda^\star;N,R)$. This completes the proof.
\end{proof}

\subsection{Radial central limit theorem}
\label{sec:app-radial-clt}

\begin{proof}[Proof of \Cref{prop:common-radial-clt}]
By \Cref{lem:common-log-gamma-clt}, independence, and the Gamma law of large numbers,
\[\left(\frac{\mathbf U_n-n/2}{\sqrt{n/2}},\frac{\mathbf V_{n,\beta_n}-\beta_n}{\sqrt{\beta_n}}\right)\overset{\mathrm d}{\longrightarrow}\gamma^1\otimes\gamma^1,\]
while $\mathbf U_n/(n/2)\overset{\mathrm P}{\longrightarrow}1$ and $\mathbf V_{n,\beta_n}/\beta_n\overset{\mathrm P}{\longrightarrow}1$. The half-angle formula gives
\[\mathbf R_{n,\beta_n}=2\operatorname{arsinh}\sqrt{\frac{\mathbf U_n}{\mathbf V_{n,\beta_n}}},\qquad \bar{\mathcal R}_{n,\beta_n}=2\operatorname{arsinh}\sqrt{\frac{n/2}{\beta_n}}.\]

\textbf{Case 1.} Suppose that $(n/2)/\beta_n\to0$. Since $\mathbf V_{n,\beta_n}/\beta_n\overset{\mathrm P}{\longrightarrow}1$ and $(\mathbf V_{n,\beta_n}-\beta_n)/\sqrt{\beta_n}$ is tight, the identity
\[\sqrt{\beta_n}\left(\sqrt{\frac{\beta_n}{\mathbf V_{n,\beta_n}}}-1\right)=-\frac{(\mathbf V_{n,\beta_n}-\beta_n)/\sqrt{\beta_n}}{\sqrt{\mathbf V_{n,\beta_n}/\beta_n}(1+\sqrt{\mathbf V_{n,\beta_n}/\beta_n})}\]
shows that its left-hand side is tight. The identity
\[2(\sqrt{\mathbf U_n}-\sqrt{n/2})=\frac{\mathbf U_n-n/2}{\sqrt{n/2}}\frac{2\sqrt{n/2}}{\sqrt{\mathbf U_n}+\sqrt{n/2}}\]
shows that this sequence is also tight. Moreover,
\[\frac{\mathbf U_n}{\beta_n}=\frac{\mathbf U_n}{n/2}\frac{n/2}{\beta_n}\overset{\mathrm P}{\longrightarrow}0,\]
and hence
\[\sqrt{\mathbf U_n}\left(\sqrt{\frac{\beta_n}{\mathbf V_{n,\beta_n}}}-1\right)=\sqrt{\frac{\mathbf U_n}{\beta_n}}\sqrt{\beta_n}\left(\sqrt{\frac{\beta_n}{\mathbf V_{n,\beta_n}}}-1\right)\overset{\mathrm P}{\longrightarrow}0.\]
Therefore,
\[\sqrt{\frac{\beta_n\mathbf U_n}{\mathbf V_{n,\beta_n}}}-\sqrt{n/2}-(\sqrt{\mathbf U_n}-\sqrt{n/2})\overset{\mathrm P}{\longrightarrow}0,\]
and $\sqrt{\beta_n\mathbf U_n/\mathbf V_{n,\beta_n}}-\sqrt{n/2}$ is tight.

The derivative of $s\mapsto2\sqrt{\beta_n}\operatorname{arsinh}(s/\sqrt{\beta_n})$ is $2/\sqrt{1+s^2/\beta_n}$. The squares of the endpoints $\sqrt{\beta_n\mathbf U_n/\mathbf V_{n,\beta_n}}$ and $\sqrt{n/2}$, divided by $\beta_n$, converge to zero because
\[\frac{\mathbf U_n}{\mathbf V_{n,\beta_n}}=\frac{\mathbf U_n}{n/2}\frac{n/2}{\beta_n}\frac{\beta_n}{\mathbf V_{n,\beta_n}}\overset{\mathrm P}{\longrightarrow}0,\qquad \frac{n/2}{\beta_n}\longrightarrow0.\]
The same holds for every intermediate point. The mean-value theorem and the tightness just proved give
\[\sqrt{\beta_n}(\mathbf R_{n,\beta_n}-\bar{\mathcal R}_{n,\beta_n})-2\left(\sqrt{\frac{\beta_n\mathbf U_n}{\mathbf V_{n,\beta_n}}}-\sqrt{n/2}\right)\overset{\mathrm P}{\longrightarrow}0.\]
Combining the last two estimates yields
\[\sqrt{\beta_n}(\mathbf R_{n,\beta_n}-\bar{\mathcal R}_{n,\beta_n})-2(\sqrt{\mathbf U_n}-\sqrt{n/2})\overset{\mathrm P}{\longrightarrow}0.\]
Finally, $2\sqrt{n/2}/(\sqrt{\mathbf U_n}+\sqrt{n/2})\overset{\mathrm P}{\longrightarrow}1$. The Gamma central limit theorem and Slutsky's theorem prove convergence to the standard Gaussian distribution.

\textbf{Case 2.} Suppose that $(n/2)/\beta_n\to\lambda\in(0,+\infty)$. Apply Taylor's theorem at $(1,1)$ to $(x,y)\mapsto2\operatorname{arsinh}\sqrt{(n/2)x/(\beta_ny)}$. Its two partial derivatives there are $\sqrt{(n/2)/(\beta_n+n/2)}$ and $-\sqrt{(n/2)/(\beta_n+n/2)}$. Since $(n/2)/\beta_n$ stays between two positive constants, the second partial derivatives are bounded by a common constant on a fixed neighborhood of $(1,1)$. On the event where $(\mathbf U_n/(n/2),\mathbf V_{n,\beta_n}/\beta_n)$ belongs to this neighborhood, the absolute Taylor remainder after multiplication by $\sqrt{\beta_n}$ is at most
\[C\sqrt{\beta_n}\left\{\frac1{n/2}\left|\frac{\mathbf U_n-n/2}{\sqrt{n/2}}\right|^2+\frac1{\beta_n}\left|\frac{\mathbf V_{n,\beta_n}-\beta_n}{\sqrt{\beta_n}}\right|^2\right\}\]
for a constant $C>0$ independent of $n$. The two squared random variables are tight, while $\sqrt{\beta_n}/(n/2)=\sqrt{\beta_n/(n/2)}/\sqrt{n/2}\to0$ and $1/\sqrt{\beta_n}\to0$. Thus the remainder converges to zero in probability. Since the neighborhood event has probability tending to one,
\[
\begin{split}
\sqrt{\beta_n}(\mathbf R_{n,\beta_n}-\bar{\mathcal R}_{n,\beta_n})&-\sqrt{\frac{\beta_n}{\beta_n+n/2}}\frac{\mathbf U_n-n/2}{\sqrt{n/2}}\\
&+\sqrt{\frac{n/2}{\beta_n+n/2}}\frac{\mathbf V_{n,\beta_n}-\beta_n}{\sqrt{\beta_n}}\overset{\mathrm P}{\longrightarrow}0.
\end{split}
\]
The two coefficients converge to $(1+\lambda)^{-1/2}$ and $\sqrt{\lambda/(1+\lambda)}$, respectively, and their squares sum to one. Joint Gaussian convergence and Slutsky's theorem again give the standard Gaussian limit.

\textbf{Case 3.} Suppose that $(n/2)/\beta_n\to+\infty$. For $x>0$,
\[
\begin{split}
\left|\frac{d}{dx}\{\operatorname{arcosh}(1+x)-\log x\}\right|&=\frac1x-\frac1{\sqrt{x(x+2)}}\\
&=\frac{2}{x^2\sqrt{1+2/x}(1+\sqrt{1+2/x})}\le\frac1{x^2}.
\end{split}
\]
The identity
\[\sqrt{\beta_n}\left(\frac{\beta_n\mathbf U_n}{(n/2)\mathbf V_{n,\beta_n}}-1\right)=\frac{\sqrt{\beta_n/(n/2)}(\mathbf U_n-n/2)/\sqrt{n/2}-(\mathbf V_{n,\beta_n}-\beta_n)/\sqrt{\beta_n}}{\mathbf V_{n,\beta_n}/\beta_n}\]
shows that its left-hand side is tight. In particular,
\[\frac{\beta_n\mathbf U_n}{(n/2)\mathbf V_{n,\beta_n}}\overset{\mathrm P}{\longrightarrow}1.\]
On the event where this ratio belongs to $[1/2,2]$, every point between $2\mathbf U_n/\mathbf V_{n,\beta_n}$ and $n/\beta_n$ is at least $(n/2)/\beta_n$. The event has probability tending to one, and the mean-value theorem and the derivative bound give
\[
\begin{split}
&\sqrt{\beta_n}\left|\left\{\operatorname{arcosh}\left(1+\frac{2\mathbf U_n}{\mathbf V_{n,\beta_n}}\right)-\log\frac{2\mathbf U_n}{\mathbf V_{n,\beta_n}}\right\}-\left\{\operatorname{arcosh}\left(1+\frac n{\beta_n}\right)-\log\frac n{\beta_n}\right\}\right|\\
&\qquad\le\frac{2\beta_n}{n/2}\left|\sqrt{\beta_n}\left(\frac{\beta_n\mathbf U_n}{(n/2)\mathbf V_{n,\beta_n}}-1\right)\right|\overset{\mathrm P}{\longrightarrow}0.
\end{split}
\]
It follows that
\[\sqrt{\beta_n}(\mathbf R_{n,\beta_n}-\bar{\mathcal R}_{n,\beta_n})-\sqrt{\beta_n}\left(\log\frac{\mathbf U_n}{n/2}-\log\frac{\mathbf V_{n,\beta_n}}{\beta_n}\right)\overset{\mathrm P}{\longrightarrow}0.\]

The events
\[\left|\frac{\mathbf U_n-n/2}{\sqrt{n/2}}\right|\le\frac{\sqrt{n/2}}2,\qquad \left|\frac{\mathbf V_{n,\beta_n}-\beta_n}{\sqrt{\beta_n}}\right|\le\frac{\sqrt{\beta_n}}2\]
have probabilities tending to one. On these events, $|\log(1+x)-x|\le x^2$ for $|x|\le1/2$ gives
\[
\begin{split}
\left|\sqrt{n/2}\log\frac{\mathbf U_n}{n/2}-\frac{\mathbf U_n-n/2}{\sqrt{n/2}}\right|&\le\frac1{\sqrt{n/2}}\left|\frac{\mathbf U_n-n/2}{\sqrt{n/2}}\right|^2,\\
\left|\sqrt{\beta_n}\log\frac{\mathbf V_{n,\beta_n}}{\beta_n}-\frac{\mathbf V_{n,\beta_n}-\beta_n}{\sqrt{\beta_n}}\right|&\le\frac1{\sqrt{\beta_n}}\left|\frac{\mathbf V_{n,\beta_n}-\beta_n}{\sqrt{\beta_n}}\right|^2.
\end{split}
\]
Both right-hand sides converge to zero in probability. The first estimate also shows that $\sqrt{n/2}\log(\mathbf U_n/(n/2))$ is tight. Since $\beta_n/(n/2)\to0$, we obtain
\[\sqrt{\beta_n}\log\frac{\mathbf U_n}{n/2}\overset{\mathrm P}{\longrightarrow}0,\qquad \sqrt{\beta_n}\log\frac{\mathbf V_{n,\beta_n}}{\beta_n}-\frac{\mathbf V_{n,\beta_n}-\beta_n}{\sqrt{\beta_n}}\overset{\mathrm P}{\longrightarrow}0.\]
Thus the logarithmic main term differs from $-(\mathbf V_{n,\beta_n}-\beta_n)/\sqrt{\beta_n}$ by a quantity converging to zero in probability. The Gamma central limit theorem, the preceding comparison, and Slutsky's theorem give the standard Gaussian limit.

For an arbitrary subsequence, compactness of $[0,+\infty]$ gives a further subsequence along which $(n/2)/\beta_n$ converges to $0$, to a point of $(0,+\infty)$, or to $+\infty$. The three cases show that
\[\sqrt{\beta_n}(\mathbf R_{n,\beta_n}-\bar{\mathcal R}_{n,\beta_n})\overset{\mathrm d}{\longrightarrow}\gamma^1\]
along this further subsequence. Thus every subsequence has a further subsequence converging in distribution to $\gamma^1$. The subsequence criterion for weak convergence gives the same convergence along the original sequence. This completes the proof.
\end{proof}

\subsection{Comparison with the truncated warped product}
\label{sec:app-warped-comparison}

\begin{proof}[Proof of \Cref{lem:gauss-warped-product-comparison}]
\textbf{Step 1.}

The bound $\mathcal A_{n,\beta_n}^2\ge2/\beta_n$ gives $\beta_n\to+\infty$. We first verify that
\[\frac{\bar{\mathcal R}_{n,\beta_n}}{\mathcal A_{n,\beta_n}}\to+\infty.\]
If $n/\beta_n\ge1$, then $\bar{\mathcal R}_{n,\beta_n}\ge\operatorname{arcosh}2$, and the claim follows from $\mathcal A_{n,\beta_n}\to0$. If $n/\beta_n\le1$, then $\cosh s\le1+s^2$ for $0\le s\le1$ gives $\operatorname{arcosh}(1+n/\beta_n)\ge\sqrt{n/\beta_n}$. Since $\mathcal A_{n,\beta_n}^2=(n/\beta_n+2)/\beta_n\le3/\beta_n$,
\[\frac{\bar{\mathcal R}_{n,\beta_n}}{\mathcal A_{n,\beta_n}}\ge\sqrt{\frac n3}\to+\infty.\]

Define
\[\Phi_n\colon J_n\times S^{n-1}\to\B^n,\qquad\Phi_n(u,\theta)=(\bar{\mathcal R}_{n,\beta_n}+\mathcal A_{n,\beta_n}u,\theta)\]
using the usual identification at the origin. By \Cref{lem:common-radial-law}, $(\Phi_n)_*(\nu_n\otimes\sigma_{n-1})=\mu_{n,\beta_n}$. Let $\widetilde d_n$ be the pullback by this map of the distance on $\mathcal A_{n,\beta_n}^{-1}\BetaBall$.

\textbf{Step 2.}

Fix $M,K>0$ and put $T\coloneqq\max\{2M,K\}$. If $|w|\le M$ and $\mathcal A_{n,\beta_n}M\le1$, the hyperbolic addition formulas give
\[
\begin{split}
\sup_{|w|\le M}\left|\frac{\sinh(\bar{\mathcal R}_{n,\beta_n}+\mathcal A_{n,\beta_n}w)}{\sinh\bar{\mathcal R}_{n,\beta_n}}-1\right|
&\le\frac{(\mathcal A_{n,\beta_n}M)^2\cosh1}{2}\\
&\quad+M\cosh1\left(\frac1{\sqrt n}+\frac{\sqrt n}{\beta_n}\right).
\end{split}
\]
Indeed, the first term on the right bounds $|\cosh(\mathcal A_{n,\beta_n}w)-1|$, while the second uses
\[\mathcal A_{n,\beta_n}\coth\bar{\mathcal R}_{n,\beta_n}=\frac{\cosh\bar{\mathcal R}_{n,\beta_n}}{\sqrt n}=\frac1{\sqrt n}+\frac{\sqrt n}{\beta_n}.\]
Since $\sqrt n/\beta_n\le\mathcal A_{n,\beta_n}$, the preceding upper bound tends to $0$.

If $0\le s\le T$ and $\mathcal A_{n,\beta_n}T/2\le1$, Taylor's theorem gives
\[0\le\frac2{\mathcal A_{n,\beta_n}}\sinh\frac{\mathcal A_{n,\beta_n}s}{2}-s\le\frac{\mathcal A_{n,\beta_n}^2T^3\cosh1}{24},\qquad0\le s-H_{\mathcal A_{n,\beta_n}}(s)\le\frac{\mathcal A_{n,\beta_n}^2s^3}{24}.\]
Put $q_n(\theta,\eta)\coloneqq H_{\mathcal A_{n,\beta_n}}(\sqrt n|\theta-\eta|)$. Rewriting \eqref{eq:common-polar-distance} and $\sinh\bar{\mathcal R}_{n,\beta_n}=\mathcal A_{n,\beta_n}\sqrt n$ by the half-angle formula gives
\[
\begin{split}
\widetilde d_n((u,\theta),(v,\eta))&=H_{\mathcal A_{n,\beta_n}}\!\Biggl(\Biggl[\left(\frac2{\mathcal A_{n,\beta_n}}\sinh\frac{\mathcal A_{n,\beta_n}|u-v|}{2}\right)^2\\
&\quad{}+\frac{\sinh(\bar{\mathcal R}_{n,\beta_n}+\mathcal A_{n,\beta_n}u)\sinh(\bar{\mathcal R}_{n,\beta_n}+\mathcal A_{n,\beta_n}v)}{\sinh^2\bar{\mathcal R}_{n,\beta_n}}\\
&\qquad{}\times\left(\frac2{\mathcal A_{n,\beta_n}}\sinh\frac{\mathcal A_{n,\beta_n}q_n(\theta,\eta)}{2}\right)^2\Biggr]^{1/2}\Biggr).
\end{split}
\]
Denote the radial ratio in the preceding display by
\[R_n(w)\coloneqq\frac{\sinh(\bar{\mathcal R}_{n,\beta_n}+\mathcal A_{n,\beta_n}w)}{\sinh\bar{\mathcal R}_{n,\beta_n}},\]
and denote its first and second components before squaring by
\[
\begin{aligned}
A_n(u,v)&\coloneqq\frac2{\mathcal A_{n,\beta_n}}\sinh\frac{\mathcal A_{n,\beta_n}|u-v|}{2},\\
B_n(u,v,\theta,\eta)&\coloneqq\sqrt{R_n(u)R_n(v)}\frac2{\mathcal A_{n,\beta_n}}\sinh\frac{\mathcal A_{n,\beta_n}q_n(\theta,\eta)}{2}.
\end{aligned}
\]
For all sufficiently large $n$, $R_n(w)$ differs from $1$ by at most $1$. On the region $|u|,|v|\le M$ and $q_n(\theta,\eta)\le K$, $A_n(u,v)$ is at most $T\cosh1$, $B_n(u,v,\theta,\eta)$ is at most $2T\cosh1$, and their Euclidean norm is at most $\sqrt5T\cosh1$. Each inverse-transform error is at most $\mathcal A_{n,\beta_n}^2T^3\cosh1/24$, while the difference from $1$ of $\sqrt{R_n(u)R_n(v)}$ is at most three times the radial-ratio bound above. Thus, in the bound below, the first term controls $|H_{\mathcal A_{n,\beta_n}}(s)-s|$ for $s=(A_n^2+B_n^2)^{1/2}$, the second controls $|A_n-|u-v||$ together with the inverse-transform part of $|B_n-q_n|$, and the third controls the radial-ratio part of $|B_n-q_n|$. The reverse triangle inequality for the Euclidean norm gives
\[
\begin{split}
&\sup_{\substack{|u|,|v|\le M\\q_n(\theta,\eta)\le K}}|\widetilde d_n((u,\theta),(v,\eta))-\delta_n((u,\theta),(v,\eta))|\\
&\le\frac{\mathcal A_{n,\beta_n}^2(\sqrt5T\cosh1)^3}{24}+\frac{\mathcal A_{n,\beta_n}^2T^3\cosh1}{8}\\
&\quad+3T\left\{\frac{(\mathcal A_{n,\beta_n}M)^2\cosh1}{2}+M\cosh1\left(\frac1{\sqrt n}+\frac{\sqrt n}{\beta_n}\right)\right\}.
\end{split}
\]
The right-hand side tends to $0$, so the pulled-back distances converge uniformly to $\delta_n$ on bounded radial and angular windows.

\textbf{Step 3.}

Fix $L,M>0$ and put $K\coloneqq L+2M+1$. The exact formula above with $u=v=0$ gives $\widetilde d_n((0,\theta),(0,\eta))=q_n(\theta,\eta)$. Since radial distance equals $|u|$, the reverse triangle inequality gives
\[\widetilde d_n((u,\theta),(v,\eta))\ge q_n(\theta,\eta)-|u|-|v|.\]
If $|u|,|v|\le M$ and $q_n(\theta,\eta)>K$, then both $\widetilde d_n\wedge L$ and $\delta_n\wedge L$ equal $L$. When $q_n(\theta,\eta)\le K$, the preceding uniform estimate applies. Therefore,
\[\sup_{\substack{|u|,|v|\le M\\\theta,\eta\in S^{n-1}}}|(\widetilde d_n\wedge L)((u,\theta),(v,\eta))-(\delta_n\wedge L)((u,\theta),(v,\eta))|\to0.\]

\textbf{Step 4.}

By \Cref{prop:common-radial-clt}, $\sqrt{\beta_n}(r-\bar{\mathcal R}_{n,\beta_n})$ is tight. The identity
\[\mathcal A_{n,\beta_n}^{-1}(r-\bar{\mathcal R}_{n,\beta_n})=\sqrt{\frac{\beta_n}{n+2\beta_n}}\,\sqrt{\beta_n}(r-\bar{\mathcal R}_{n,\beta_n})\]
and the coefficient bound $1/\sqrt2$ show that $(\nu_n)$ is tight. Given $\varepsilon>0$, choose $M>0$ such that $\nu_n([-M,M])\ge1-\varepsilon$ for all sufficiently large $n$. The already established convergence $\bar{\mathcal R}_{n,\beta_n}/\mathcal A_{n,\beta_n}\to+\infty$ also gives $[-M,M]\subset J_n$ eventually. Under the diagonal coupling induced by the common product measure, the distortion of the truncated metrics on $[-M,M]\times S^{n-1}$ is at most $\varepsilon$ for all sufficiently large $n$, and this set has mass at least $1-\varepsilon$. By \Cref{def:common-box-concentration},
\[\Box\bigl((\mathcal A_{n,\beta_n}^{-1}\BetaBall)^{(L)},Y_n^{(L)}\bigr)\le\varepsilon.\]
Letting $\varepsilon\downarrow0$ proves the assertion. This completes the proof.
\end{proof}

\section*{Acknowledgments}

The author would like to thank Professor Takashi Shioya for many helpful suggestions and guidance. The author used Claude, GPT-5.5, and GPT-5.6-series Codex models as AI-assisted tools in preparing this manuscript. The author reviewed and revised the mathematical content and takes full responsibility for the final manuscript.

\section*{Statements and Declarations}

\noindent\textbf{Funding.}
During the preparation of this manuscript, the author received financial support from the Tohoku University Advanced Graduate School LEAP Program.

\medskip
\noindent\textbf{Competing Interests.}
The author has no relevant financial or non-financial interests to disclose.

\end{document}